\documentclass[onefignum,onetabnum,final]{siam/siamonline171218}
\usepackage{amssymb,algorithm,algorithmic,float,graphicx,mathtools}

\usepackage{pgfplots}
  \pgfplotsset{compat=newest}
  \usetikzlibrary{plotmarks}
  \usetikzlibrary{arrows.meta}
  \usepgfplotslibrary{patchplots}
  \usepackage{grffile}
  \usepackage{amsmath}

  \newlength\figureheight
  \newlength\figurewidth

\newcommand{\mc}[1]{\mathcal{#1}}
\newcommand{\mb}[1]{\mathbb{#1}}
\newcommand{\bmat}[1]{\begin{bmatrix} #1 \end{bmatrix}}
\newcommand{\mat}[1]{\mathbf{#1}}
\newcommand{\ten}[1]{\boldsymbol{\mathcal{#1}}}
\newcommand{\tenh}[1]{\widehat{\boldsymbol{\mc{#1}}}}
\newcommand{\rank}{\mathsf{rank}}
\newcommand{\range}{\mathcal{R}}

\usepackage{color}

\newtheorem{remark}{Remark}

\graphicspath{ {Images/} }

\title{Randomized Algorithms for Low-rank Tensor Decompositions in the Tucker Format\thanks{\funding{The authors would like to  acknowledge the support of the National Science Foundation through the grants DMS-1821148 (M.E.K) and DMS-1821149 (R.M. and A.K.S).}}}
\author{Rachel Minster\thanks{Department of Mathematics, North Carolina State University, \email{rlminste@ncsu.edu}} \and Arvind K. Saibaba\thanks{Department of Mathematics, North Carolina State University, \email{asaibab@ncsu.edu}} \and Misha E. Kilmer\thanks{Department of Mathematics, Tufts University, \email{misha.kilmer@tufts.edu}}}

\begin{document}
\maketitle

\begin{abstract}
Many applications in data science and scientific computing involve large-scale datasets that are expensive to store and compute with, but can be efficiently compressed and stored in an appropriate tensor format.  In recent years, randomized matrix methods have been used to efficiently and accurately compute low-rank matrix decompositions. Motivated by this success, we focus on developing randomized algorithms for tensor decompositions in the Tucker representation. Specifically, we present randomized versions of two well-known compression algorithms, namely, HOSVD and STHOSVD. We present a detailed probabilistic analysis of the error of the randomized tensor algorithms. We also develop variants of these algorithms that tackle specific challenges posed by large-scale datasets. The first variant adaptively finds a low-rank representation satisfying a given tolerance and it is beneficial when the target-rank is not known in advance. The second variant preserves the structure of the original tensor, and is beneficial for large sparse tensors that are difficult to load in memory. We consider several different datasets for our numerical experiments: synthetic test tensors and realistic applications such as the compression of facial image samples in the Olivetti database and word counts in the Enron email dataset.
\end{abstract}

\section{Introduction}
Tensors, or multi-way arrays, appear in a wide range of applications such as signal processing; neuroscientific applications such as Electroencephalography; data mining; seismic data processing; machine learning applications such as facial recognition, handwriting digit classification, and latent semantic indexing; imaging; astronomy; and uncertainty quantification.  For example, a database of gray scale images constitutes a third order array when each image is stored as a matrix, while a numerical simulation of system of a partial differential equations (PDEs) in three-dimensional space when tracking several parameters over time yields a five-dimensional dataset. Often, these datasets are treated as matrices rather than as tensors, suggesting that additional structure that could be leveraged for gaining insight and lowering computational cost is often underutilized and undiscovered. 

A key step in processing and studying these datasets involves a compression step either to find an economical representation in memory, or to find principal directions of variability. While working with tensors there are many possible formats one may consider, and each format is equipped with a different notion of compression and rank. Examples of tensor formats include CANDECOMP/PARAFAC (CP), Tucker, Hierarchical Tucker, and Tensor Train, all of which have their respective benefits (see surveys~\cite{kolda2009tensor,grasedyck2013literature,cichocki2016tensor,cichocki2017tensor}). The CP format which represents a tensor as  a sum of rank$-1$ outer products gives a compact and unique (under certain conditions) representation. Tucker generally finds a better fit for data by estimating the subspace of each mode, while Hierarchical-Tucker and Tensor Train are useful for very high-dimensional tensors.  In this paper, we focus on the Tucker representation which is known for its favorable compression properties in a modest number of dimensions (3-7 modes). Given a multilinear rank $\mat{r}$, the Tucker form finds a rank $\mat{r}$ representation of a tensor as a product of a core tensor and factor matrices typically having orthonormal columns.  Popular algorithms for compression in the Tucker format can be found in \cite{de2000multilinear,vannieuwenhoven2012new,de2000best}, and a survey of approximation techniques can be found in \cite{grasedyck2013literature}. Also, depending on how small the target rank for an approximation is compared to the original dimensions, high compression can be achieved. If the data is such that this is not possible, representing it in the Tucker format can still give insight into its principal directions, since Tucker is also a form of higher-dimensional principal component analysis (PCA).

In recent years, randomized matrix algorithms have gained popularity for developing low-rank matrix approximations (see the review~\cite{halko2011finding,mahoney2011randomized,drineas2016randnla}). These algorithms are easy to implement, computationally efficient for a wide range of matrices (e.g. sparse matrices, matrices that can be accessed only via matrix-vector products, and dense matrices that are difficult to load in memory), and have accuracy comparable  with non-randomized algorithms. There is also well-developed error analysis applicable to several classes of random matrices for randomized algorithms. Even more recently, randomized algorithms have been developed for tensor decompositions (see below for a detailed review).  In this paper, we make several contributions by proposing new algorithms that are accompanied by  rigorous error analysis. The results in this paper enable the efficient computation of low-rank Tucker decompositions across a wide range of applications.  The randomized algorithms developed here exploit the structure  of tensor decompositions, while the analysis provides insight into the choice of parameters that control the tradeoff between accuracy and computational efficiency.

\paragraph{Contributions and Contents} In \cref{sec:background}, we first review  the necessary background information on tensors and randomized algorithms. Then, in \cref{sec:rand_algs}, we present analyses of randomized versions of HOSVD and STHOSVD (proposed in \cite{zhou2014decomposition} and \cite{che2018randomized} respectively). Our contributions on this front include the probabilistic analysis of the randomized versions of these algorithms in expectation, as well as analysis of the associated computational costs. In \cref{sec:adapt} we present adaptive randomized algorithms to compute low-rank tensor decompositions to be used in applications where the target rank is not known beforehand.  In \cref{sec:sparse}, we present a new randomized compression algorithm for large tensors, which produces a low-rank decomposition whose core tensor has entries taken directly from the tensor of interest. In this sense, the core tensor preserves the structure (e.g., sparsity, non-negativity) of the original tensor. For sparse tensors, our algorithm has  the added benefit that the intermediate and final decompositions can be stored efficiently, thus enabling the computation of low-rank tensor decompositions of large, sparse tensors. To supplement the algorithm, we provide a probabilistic analysis in expectation.  Finally, in \cref{sec:num_results},  we test the performance of all algorithms on several synthetic tensors and real-world datasets, and discuss the performance of the proposed bounds.

\paragraph{Related Work}
Several randomized algorithms have been proposed for computing low-rank tensor decompositions, e.g., Tucker format \cite{che2018randomized,zhang2016randomized,kressner2017recompression,malik2019fast,tsourakakis2010mach,erichson2017randomized,zhou2014decomposition}, CP format~\cite{erichson2017randomized,battaglino2018practical,biagioni2015randomized,vervliet2016randomized},  t-product~\cite{zhang2016randomized}, tensor networks~\cite{batselier2018computing}, and Tensor Train format~\cite{che2018randomized,huber2017randomized}.   Our work is most similar to \cite{che2018randomized,erichson2017randomized,zhou2014decomposition}. The algorithm for randomized HOSVD is presented in \cite{zhou2014decomposition}, and the corresponding analysis is presented in \cite{erichson2017randomized} (both unpublished manuscripts).    Randomized and adaptive versions of the STHOSVD were proposed and analyzed in \cite{che2018randomized}, but our manuscript uses a different distribution of random matrices (see \cref{sec:rand_algs} for a justification of our choice), and provides bounds in expectation. To our knowledge, our proposed algorithm for producing structure-preserving tensor decompositions and the corresponding error analysis are novel. Related to this algorithm is the CUR-type decomposition for tensors proposed in~\cite{saibaba2017hoid,drineas2007randomized}. In contrast, our algorithm produces decompositions in which the core tensor (rather than the factor matrices, in the aforementioned references) retains entries from the original tensor. 

\section{Background}\label{sec:background}
In this section, we introduce the necessary background information for working with tensors, and review the standard compression algorithms.  We also discuss the optimal approximation of a tensor for comparison purposes. Finally, we review the relevant background for randomized matrix algorithms, specifically the randomized SVD.

\subsection{Notation and preliminaries}
We denote a $d$-mode tensor $ \ten{X} \in \mb{R}^{I_1 \times \dots \times I_d}$ with entries 
$$x_{i_1,\dots,i_d} \qquad 1 \leq i_j \leq I_j \qquad j=1,\dots,d.$$ 
A tensor can be ``unfolded'' into a matrix by reordering the elements, and this process is known as {\em matricization}. There are $d$ different unfoldings for a $d$-mode tensor.  Each mode-$j$ unfolding arranges the resulting matrix such that the columns are the mode-$j$ fibers of the tensor. The mode-$j$ unfolding is denoted as  $\mat{X}_{(j)} \in \mb{R}^{I_j \times (\prod_{k \neq j} I_k)}$ for $j= 1,\dots,d$.

\paragraph{Tensor product} The tensor product (or mode product) is a fundamental operation for multiplying a tensor by a matrix. Given a matrix $\mat{A} \in \mb{R}^{K \times I_j}$, the mode-$j$ product of a tensor $\ten{X}$ with $\mat{A}$ is denoted $\ten{Y} = \ten{X} \times_j \mat{A}$, and has dimension $\ten{Y}\in \mb{R}^{I_1 \times \dots I_{j-1} \times K \times I_{j+1} \times \dots \times I_d}$. More specifically, the product can be expressed in terms of the entries of the tensor as  
$$\ten{Y}_{i_1,\dots,i_{j-1},k,i_{j+1},\dots,i_d} = \sum_{i_j=1}^{I_j} x_{i_1,\dots,i_d} a_{ki_j} \qquad 1 \leq k \leq K \qquad j=1,\dots,d.$$
The tensor product can also be expressed as the product of two matrices. That is, we can write $\ten{Y}_{(j)} = \mat{A} \mat{X}_{(j)}$ for $j=1,\dots,d$. The following lemma will be useful in our analysis. 
\begin{lemma}\label{lem:proj} Let $\ten{X} \in \mb{R}^{I_1 \times \dots \times I_d}$ and let $\mat{\Pi}_j \in \mb{R}^{I_j \times I_j}$ be a sequence of $d$ orthogonal projectors. Then for $j = 1,2,\dots,d$,
\[ \| \ten{X} - \ten{X}\bigtimes_{j=1}^d \mat\Pi_j\|_F^2  = \sum_{j=1}^d\|\ten{X} \bigtimes_{i=1}^{j-1} \mat{\Pi}_i \times_j (\mat{I} - \mat\Pi_j) \|_F^2  \leq \sum_{j=1}^d\|\ten{X} - \ten{X}\times_j \mat\Pi_j\|_F^2. \]
\end{lemma}
The proof of this lemma can be found in \cite[Theorem 5.1]{vannieuwenhoven2012new}.

\paragraph{Tucker representation} The Tucker format of a tensor $\ten{X}$ of rank $(r_1,\dots, r_d)$ consists of a core tensor $\ten{G} \in \mb{R}^{r_1 \times \dots \times r_d}$ and factor matrices $\{\mat{A}_j\}_{j=1}^d$ with each $\mat{A}_j \in \mb{R}^{I_j \times r_j}$. For short, it is written as $[\ten{G}; \mat{A}_1,\dots,\mat{A}_d]$ and represents $\ten{X} = \ten{G} \bigtimes_{j=1}^d \mat{A}_j$.  

Note that storing a tensor in Tucker form is beneficial as it requires less storage than a full tensor when the target rank is significantly less than the original dimension.  For a $d$-mode tensor $\ten{X} \in \mb{R}^{I \times I \times \dots \times I}$ and target rank $(r,r,\dots,r)$ with $r \ll I$, the cost of storing the Tucker form of $\ten{X}$ is $\mc{O}(r^d+drI)$, compared to $\mc{O}(I^d)$ for a full tensor.

\paragraph{Kronecker products} The Kronecker product of two matrices $\mat{A} \in \mb{R}^{m \times n}$ and $\mat{B} \in \mb{R}^{k \times \ell}$ is 
$$\mat{A} \otimes \mat{B} = \bmat{a_{11}\mat{B} & a_{12}\mat{B} & \cdots & a_{1n}\mat{B} \\ a_{21}\mat{B} & a_{22}\mat{B} & \cdots & a_{2n}\mat{B} \\ \vdots & \vdots & \ddots & \vdots \\ a_{m1}\mat{B} & a_{m2}\mat{B} & \cdots & a_{mn}\mat{B}} \in \mb{R}^{mk \times n\ell}.$$  
We also note some properties of Kronecker products that will be useful in our analysis, namely
\begin{equation*}
\begin{aligned}
(\mat{A} \otimes \mat{B})(\mat{C} \otimes \mat{D}) &= \mat{AC} \otimes \mat{BD} \\
(\mat{A} \otimes \mat{B})^\top &= \mat{A}^\top \otimes \mat{B}^\top.
\end{aligned}
\end{equation*}
Kronecker products are also useful for expressing tensor mode products in terms of matrix-matrix multiplications. Suppose $\ten{Y} = \ten{X} \bigtimes_{j=1}^d \mat{A}_j$, then 
\begin{equation}\label{eqn:kron}
    \mat{Y}_{(j)} = \mat{A}_j \mat{X}_{(j)}(\mat{A}_d^\top \otimes \mat{A}_{d-1}^\top \otimes \dots \otimes \mat{A}_{j+1}^\top \otimes \mat{A}_{j-1}^\top \otimes \dots \otimes \mat{A}_1^\top).
\end{equation}

\subsection{HOSVD/STHOSVD}\label{ssec:hosvd}
The {\em Higher Order SVD} (HOSVD) and {\em Sequentially Truncated Higher Order SVD} (STHOSVD) are two popular algorithms for computing low-rank tensor decompositions in the Tucker format.  Given a $d$-mode tensor $\ten{X} \in \mb{R}^{I_1 \times \dots \times I_d}$ and target rank $\mat{r} = (r_1,\dots,r_d)$, both algorithms give a compressed representation for the tensor in the Tucker format 
$\ten{X} \approx [\ten{G};\mat{A}_1, \dots, \mat{A}_d],$
where $\ten{G} \in \mb{R}^{r_1 \times \dots \times r_d}$ is the core tensor of reduced rank, and $\{\mat{A}_j\}_{j=1}^d$ are factor matrices such that $\ten{X} = \ten{G} \bigtimes_{j=1}^d \mat{A}_j$.  The factor matrices $\mat{A}_j$ all have orthonormal columns, and each $\mat{A}_j \in \mb{R}^{I_j \times r_j}$.  

\paragraph{HOSVD} In the HOSVD algorithm, each mode is handled separately. The factor matrix $\mat{A}_j$ is formed from the first $r_j$ left singular vectors of $\mat{X}_{(j)}$.  Once all factor matrices $\mat{A}_j$ are found, the core tensor is formed by $\ten{G} = \ten{X} \bigtimes_{j=1}^d \mat{A}_j^\top$. The error in approximating $\ten{X}$ using the HOSVD depends on the error in each mode, as shown in the following theorem, the proof of which can be found in~\cite[Theorem 5.1]{vannieuwenhoven2012new}.  
\begin{theorem}
\label{lem:mode_err}
Let $\tenh{X} = [\ten{G};\mat{A}_1,\dots,\mat{A}_d]$ be the rank-$\mat{r}$  approximation to $d$-mode tensor $\ten{X} \in \mb{R}^{I_1\times \dots  \times I_d}$ using the HOSVD algorithm.  Then
$$\| \ten{X} - \tenh{X} \|_F^2 \leq \sum_{j=1}^d \| \ten{X}\times_j (\mat{I} - \mat{A}_j \mat{A}_j^\top) \|_F^2 = \sum_{j=1}^d  \sum_{i = r_j+1}^{I_j} \sigma_i^2(\mat{X}_{(j)}).$$
\end{theorem}
This theorem says that the error in the rank $\mat{r}$ approximation of the tensor $\ten{X}$ computed using the HOSVD is the sum of squares of the discarded singular values from each mode unfolding. To simplify the upper bound, we introduce the notation
\vspace{-.2cm}
\begin{equation}\label{eqn:delta}
\Delta_j^2 (\ten{X})\equiv  \sum_{i = r_j+1}^{I_j} \sigma_i^2(\mat{X}_{(j)}) \qquad j=1,\dots,d.
\vspace{-.3cm}
\end{equation}
With this notation, the error in the HOSVD satisfies $\| \ten{X} - \tenh{X} \|_F \leq \left(\sum_{j=1}^d \Delta_j^2(\ten{X})\right)^{1/2}$.

\paragraph{STHOSVD} An alternative to the HOSVD is the {\em sequentially truncated HOSVD} \break (STHOSVD) algorithm which also produces a compressed representation in the Tucker format. In contrast to HOSVD which processes the modes independently, the STHOSVD processes the modes sequentially. This makes the order in which the modes are processed important, since we may obtain different approximations by using different processing orders. At each stage, the core tensor is unfolded (initialized as the tensor $\ten{X}$) and the factor matrix $\mat{A}_j$ is obtained by taking the first $r_j$ singular vectors. A new core tensor is obtained by projecting the core tensor onto the subspace spanned by the columns of the factor matrix. A characteristic feature of the STHOSVD is that the core tensor shrinks at each iteration thus making the later modes cheaper to compute. 

Given a tensor $\ten{X}$ and the processing order $\rho = [1,2,\dots,d]$, the rank-$\mat{r}$ STHOSVD approximation of $\ten{X}$ is $\tenh{X} = [\ten{G}; \mat{A}_1,\dots, \mat{A}_d]$, where each factor matrix $\mat{A}_j \in \mb{R}^{I_j \times r_j}$ has orthonormal columns, and the core tensor $\ten{G} \in \mb{R}^{r_1 \times r_2 \times \dots \times r_d}$ is defined as $\ten{G} = \ten{X} \bigtimes_{j=1}^d \mat{A}_j^\top.$
At the $j$-th step of the process, we have a partially truncated core tensor $\ten{G}^{(j)} \in \mb{R}^{r_1,\dots, r_j,I_{j+1},\dots,I_d}$ defined as
$\ten{G}^{(j)} = \ten{X} \bigtimes_{i=1}^j \mat{A}_i^\top.$
Then the $j$-th partial approximation, of rank $(r_1,\dots,r_j,I_{j+1},\dots,I_d)$, can be defined as
$\tenh{X}^{(j)} = \ten{G}^{(j)} \bigtimes_{i = 1}^j \mat{A}_i.$ The algorithm is initialized with $\tenh{X}^{(0)} = \ten{X}$.

The approximation error in this case is the sum of errors in the successive approximations, and has the same upper bound as that of HOSVD.  This is shown in the following theorem, which assumes that the processing order is $\rho = [1,2,\dots,d]$.  If a different processing order is taken, the bound will remain the same, so this assumption is taken for ease of notation.  The proof of this theorem can be found in \cite[Theorem 6.5]{vannieuwenhoven2012new}.
\begin{theorem}
\label{lem:st_mode_err}
Let $\tenh{X} = [\ten{G}; \mat{A}_1,\dots,\mat{A}_d]$ be the rank-$\mat{r}$ STHOSVD approximation to $d$-mode tensor $\ten{X}$ with $\rho = [1,2,\dots,d]$ the processing order of modes. Then
\vspace{-.15cm}
\begin{equation*}
\begin{aligned}
\| \ten{X} - \tenh{X} \|_F^2 = & \> \sum_{j=1}^d \| \tenh{X}^{(j-1)}-\tenh{X}^{(j)} \|_F^2 \leq  \> \sum_{j=1}^d \| \ten{X} \times_j (\mat{I} - \mat{A}_j\mat{A}_j^\top) \|_F^2 = \sum_{j=1}^d \Delta_j^2(\mat{X}_{(j)}).
\end{aligned}
\end{equation*}
\end{theorem}
 The computational cost of the STHOSVD is lower than the HOSVD, which was established in \cite{vannieuwenhoven2012new} but is also reviewed in \cref{ssec:comp}. Although both the error in the HOSVD and the STHOSVD satisfy the same upper bound, it is not clear which algorithm has a lower error. There is strong numerical evidence to suggest that STHOSVD typically has a lower error, although counterexamples to this claim have been found~\cite{vannieuwenhoven2012new}. For these reasons, STHOSVD is preferable to HOSVD since it has a lower cost and the same worst case error bound. A downside to STHOSVD is that the processing order $\rho$ has to be determined in advance; some heuristics for this choice are given in~\cite{vannieuwenhoven2012new}.

\subsection{Best Approximation}
We would like to find an optimal rank-$\mat{r}$ approximation of a given tensor $\ten{X}$, which we will call $\tenh{X}_\text{opt}$.  Let $\mc{S} = \{\ten{Y}\in\mb{R}^{I_1 \times I_2 \times \dots \times I_d} : \text{rank}(\mat{Y}_{(j)}) \leq r_j, j = 1,\dots,d \}$.  Then $\tenh{X}_\text{opt}$ is an optimal tensor which satisfies the condition 
$$\min_{\ten{Y} \in \mc{S}} \| \ten{X} - \ten{Y} \|_F = \|\ten{X} - \tenh{X}_\text{opt} \|_F.$$
 The Eckart-Young theorem \cite{eckart1936approximation} states that the optimal rank-$r$ approximation to a matrix $\mat{A}$ can be constructed using the SVD truncated to rank-$r$.  Unfortunately, an analog of this result for Tucker forms does not exist in higher dimensions. The existence of $\tenh{X}_\text{opt}$ is guaranteed by~\cite[Theorem 10.8]{hackbusch2012tensor}. In general, computing $\tenh{X}_\text{opt}$ requires the solution of an optimization problem. In~\cite{de2000best}, the higher order orthogonal iteration (HOOI), was proposed to compute the ``best'' approximation by generating a sequence of iterates by cycling through the modes sequentially. 
Because the HOOI algorithm requires repeated computations with the tensor $\ten{X}$, its implementation for large-scale tensors is challenging because of the overwhelming computational cost.   
 Although neither the HOSVD nor the STHOSVD produce an optimal rank-$\mat{r}$ approximation, they do satisfy the inequality
\begin{equation}\label{eqn:xbest}
\| \ten{X} - \tenh{X} \|_F \leq \sqrt{d} \|\ten{X} - \tenh{X}_\text{opt} \|_F.
\end{equation}
The proofs are available for the HOSVD (Theorem 10.3) and STHOSVD (Theorem 10.5) in~\cite{hackbusch2012tensor}. The proof requires the observation that 
\begin{equation}\label{eqn:xbestdelta}
\Delta_j(\ten{X}) \leq \|\ten{X}-\tenh{X}_\text{opt}\|_F \qquad j=1,\dots,d.
\end{equation}
We highlight this inequality since it will be important for our subsequent analysis. The inequality~\cref{eqn:xbest} suggests that the outputs of the HOSVD and the STHOSVD are accurate for low dimensions and can be employed in three different ways: either as approximations to $\tenh{X}_\text{opt}$, as starting guesses to the HOOI algorithm, or to fit CP models.

\subsection{Randomized SVD}
It is well-known that computing the full SVD of a matrix  costs $\mc{O}(mn^2)$, assuming $m \geq n$. When the dimensions of a matrix $\mat{X} \in \mb{R}^{m \times n}$ are very large, the computational cost of a full SVD may be prohibitively expensive. Randomized SVD, popularized by~\cite{halko2011finding}, is a computationally efficient way to compute a rank-$r$ approximation of a matrix $\mat{X}$.  Assuming that $\mat{X}$ is approximately low-rank or has singular values that decay rapidly, the randomized SVD delivers a good low-rank representation of the matrix. Compared to the full SVD, the randomized SVD is much more computationally efficient across a wide-range of matrices.

We now describe the randomized SVD algorithm for computing a rank-$r$ approximation to $\mat{X}$, where the target rank $1 \leq r \leq \rank(\mat{X})$. The randomized algorithm has two distinct stages---the range finding stage, and the postprocessing stage to compute the low-rank approximation. In the range finding stage, we  first draw  a random matrix $\mat{\Omega} \in \mb{R}^{n \times (r+p)}$, where $r$ is the desired target rank, and $p$ is a small nonnegative integer used as an oversampling parameter. While many choices for the distribution of $\mat{\Omega}$ are  possible, in this paper, we use the Gaussian distribution, i.e., the entries $\mat{\Omega}_{ij}$ are independent and identically distributed $\mc{N}(0,1)$ random variables. Then, we compute the matrix  $\mat{Y} = \mat{X}\mat{\Omega}$ whose columns consist of random linear combinations  of the columns of $\mat{X}$.  By taking a thin QR factorization $\mat{Y} = \mat{Q}\mat{R}$, we get a matrix $\mat{Q}$ with orthonormal columns whose span approximates the range of $\mat{X}$.  If the matrix $\mat{X}$ is approximately rank $r$, $\range(\mat{Q})$ is a good approximation for $\range(\mat{X})$ as the range of $\mat{X}$ is characterized by just the first $r$ left singular vectors.  We can then express $\mat{X} \approx \mat{Q}\mat{Q}^\top \mat{X}$.  In the second stage, we convert the low rank representation into the SVD format. To this end, we compute the thin SVD of $\mat{B} = \mat{Q}^\top \mat{X}$, giving $\mat{B} = \widehat{\mat{U}}_\mat{B}\widehat{\mat{\Sigma}}\widehat{\mat{V}}^\top$.  We truncate this representation to rank $r$ by only retaining the first $r$ diagonal elements of $\mat{\Sigma}$ and drop the corresponding columns from $\widehat{\mat{U}}_\mat{B}$ and $\widehat{\mat{V}}$. Finally, we compute $\widehat{\mat{U}} = \mat{Q}\widehat{\mat{U}}_\mat{B}$, and obtain the low-rank approximation $\widehat{\mat{X}} = \widehat{\mat{U}}\widehat{\mat{\Sigma}}\widehat{\mat{V}}^\top$.  \cref{alg:randsvd} summarizes the process. 

\begin{algorithm}[H]
\begin{algorithmic}[1]
\REQUIRE matrix $\mat{X} \in \mb{R}^{m \times n}$, target rank $r$, \\ $\quad$ oversampling parameter $p \geq 0$ such that $r+p \leq \min \{m,n\}$, \\ $\quad$ Gaussian random matrix $\mat{\Omega} \in \mb{R}^{n \times (r+p)}$
\ENSURE $\widehat{\mat{U}} \in \mb{R}^{m \times r}$, $\widehat{\mat{\Sigma}} \in \mb{R}^{r \times r}$, and $\widehat{\mat{V}} \in \mb{R}^{n \times r}$ such that $\mat{X} \approx \widehat{\mat{U}}\widehat{\mat{\Sigma}} \widehat{\mat{V}}^\top$
\STATE Multiply $\mat{Y} \leftarrow \mat{X}\mat{\Omega}$
\STATE Thin QR factorization $\mat{Y} = \mat{Q}\mat{R}$
\STATE Form $\mat{B} \leftarrow \mat{Q}^\top \mat{X}$
\STATE Calculate thin SVD $\mat{B} = \widehat{\mat{U}}_\mat{B} \widehat{\mat{\Sigma}} \widehat{\mat{V}}^\top$
\STATE Form $\widehat{\mat{U}} \leftarrow \mat{Q} \widehat{\mat{U}}_\mat{B}(:\,,1:r)$
\STATE Compress $\widehat{\mat{\Sigma}} \leftarrow {\widehat{\mat{\Sigma}}}(1:r,1:r)$, and $\widehat{\mat{V}} \leftarrow \widehat{\mat{V}}(:\,,1:r)$ 
\end{algorithmic}
\caption{$[\widehat{\mat{U}},\widehat{\mat{\Sigma}},\widehat{\mat{V}}] =\text{RandSVD}(\mat{X},r,p,\mat{\Omega})$}
\label{alg:randsvd}
\end{algorithm}
We briefly review the computational cost of the RandSVD algorithm. Let $T$ denote the computational cost of a matrix-vector product (denoted matvec) involving the matrix $\mat{X}$. Then, the cost of the algorithm is 
\[ \text{Cost} = 2(r+p)\mc{O}(\mathsf{nnz}(\mat{X})) + \mc{O}(r^2(m+n)),\]
 where $\mathsf{nnz}$ denotes the number of nonzeros of $\mat{X}$. 

An error bound for~\cref{alg:randsvd} in the Frobenius norm is presented below, and we will use this result frequently in our analysis.  This theorem and its proof can be found in~\cite[Theorem 3]{zhang2016randomized}.

\begin{theorem} Let $\mat{X} \in \mb{R}^{m \times n}$, and $\mat{\Omega} \in \mb{R}^{n \times (r+p)}$ be a Gaussian random matrix.  Suppose $\mat{Q}$ is obtained from~\cref{alg:randsvd} with inputs target rank $r$ and oversampling parameter $p \geq 2$ such that $r+p \leq \min\{m,n\}$, and let $\mat{B}_r$ be the rank-$r$ truncated SVD of $\mat{Q}^\top \mat{X}$.  Then, the error in expectation satisfies
$$\mb{E}_{\mat{\Omega}}\| \mat{X}-\mat{Q}\mat{Q}^\top \mat{X} \|_F^2 \leq \mb{E}_{\mat{\Omega}} \| \mat{X} - \mat{Q}\mat{B}_r \|_F^2 \leq \left(1+\frac{r}{p-1}\right) \sum_{j=r+1}^{\min\{m,n\}} \sigma_j^2(\mat{X}).$$
\label{thm:randsvd_err}
\end{theorem}
\begin{remark}
We will use two slightly different formulations of this theorem in our later results. Instead of $\| \mat{X} - \mat{Q}\mat{B}_r \|_F^2$, we will use $\| \mat{X} - \widehat{\mat{U}}\widehat{\mat{U}}^\top \mat{X} \|_F^2$. It is straightforward to show the equivalence between the two forms, see \cite[section 5.3]{saibaba2019randomized} for the explicit details. We will also use 
\begin{equation}\label{eqn:nonsquared_err}
\mb{E}_{\mat{\Omega}}\| \mat{X} - \mat{QQ}^\top \widehat{\mat{X}} \|_F \leq  \sqrt{1+\frac{r}{p-1}} \left(\sum_{j=r+1}^{\min\{m,n\}} \sigma_j^2(\mat{X})\right)^{1/2},
\end{equation}
which can be obtained by applying H{\"o}lder's inequality~\cite[Theorem 23.10]{jacod2012probability} to the stated result in the theorem.
\end{remark}

We could use other distributions for the random matrix $\mat{\Omega}$, but probabilistic bounds in expectation like \cref{thm:randsvd_err} do not exist for any distribution other than Gaussian. There are large deviation bounds for other distributions but we do not consider them here.

\section{Randomized HOSVD/STHOSVD} \label{sec:rand_algs}
In this section, we present randomized algorithms that are modified versions of the HOSVD and STHOSVD. We also develop rigorous error analysis and compare the two  algorithms in terms of computational cost.

\subsection{Algorithms}
As mentioned earlier, the HOSVD algorithm first computes an SVD of each mode unfolding to construct the factor matrices, which are then used to form the core tensor. Computing this decomposition for large, high-dimensional tensors can be prohibitively expensive. To address this computation cost, we replace a full SVD of each mode unfolding with a randomized SVD of each mode unfolding to construct the factor matrix. The procedure to compute the core tensor remains unchanged. This is reflected in~\cref{alg:rhosvd} and we call this the R-HOSVD algorithm. To our knowledge, this was first proposed in \cite{zhou2014decomposition}. 
\begin{algorithm}[!ht]
\begin{algorithmic}[1]
\REQUIRE $d$-mode tensor $\ten{X} \in \mb{R}^{I_1 \times I_2 \times \dots \times I_d} $, target rank vector $\mat{r} \in \mb{N}^d$, \\ $\quad$ oversampling parameter $p \geq 0$ such that $r_j+p \leq \min \{I_j,\prod_{i\neq j} I_i\}$ for $j = 1,\dots,d$
\ENSURE $\tenh{X} = [\ten{G}; \mat{A}_1,\dots,\mat{A}_d]$
\FOR{$j = 1:d$}
\STATE Draw random Gaussian matrix $\mat{\Omega}_j \in \mb{R}^{\prod_{i \neq j} I_i \times (r_j+p)}$
\STATE $[\widehat{\mat{U}},\widehat{\mat{\Sigma}},\widehat{\mat{V}}] = $ RandSVD$(\mat{X}_{(j)},r_j,p,\mat{\Omega}_j)$
\STATE Set $\mat{A}_j \leftarrow \widehat{\mat{U}}$
\ENDFOR
\STATE Form $\ten{G} = \ten{X} \bigtimes_{j=1}^d \mat{A}_j^\top$
\end{algorithmic}
\caption{Randomized HOSVD}
\label{alg:rhosvd}
\end{algorithm}
The randomized version of STHOSVD is obtained in a similar way; at each step, the SVD  of the unfolded core tensor is replaced with a randomized SVD. This is shown in \cref{alg:rsthosvd} (we call this R-STHOSVD) and is similar to the algorithm proposed in \cite{che2018randomized}. The major difference of \cref{alg:rsthosvd} compared to \cite{che2018randomized} is the distribution of random matrices $\{\mat\Omega_j\}_{j=1}^d$. The authors in \cite{che2018randomized} advocate $\mat\Omega_j$ constructed as a Khatri-Rao product of Gaussian random matrices as opposed to standard Gaussian random matrices which we use. The main reason for avoiding standard Gaussian random matrices appear to be because of the high storage costs; however, we note that the matrices $\mat\Omega_j$ need not be stored explicitly. Its entries may be generated on-the-fly, either column-wise or in appropriately sized blocks. We next analyze the error in the decompositions produced using the R-HOSVD and the R-STHOSVD algorithms.
\begin{algorithm}[!ht]
\begin{algorithmic}[1]
\REQUIRE $d$-mode tensor $\ten{X} \in \mb{R}^{I_1 \times I_2 \times \dots \times I_d}$, processing order $\rho$, target rank vector $\mat{r} \in \mb{N}^d$, \\ $\quad$ oversampling parameter $p \geq 0$ such that $r_j+p \leq \min \{I_j,\prod_{i\neq j} I_i\}$ for $j=1,\dots,d$
\ENSURE $\tenh{X} = [\ten{G}; \mat{A}_1,\dots,\mat{A}_d]$ 
\STATE Set $\ten{G} = \ten{X}$
\FOR {$j = 1:d$}
\STATE Draw random Gaussian matrix $\mat{\Omega}_{\rho_j} \in \mb{R}^{\prod_{\rho_i \neq \rho_j} I_{\rho_i} \times (r_{\rho_j}+p)}$
\STATE $[\widehat{\mat{U}},\widehat{\mat{\Sigma}},\widehat{\mat{V}}] =$ RandSVD$(\mat{G}_{(\rho_j)},r_{\rho_j},p,\mat{\Omega}_{\rho_j})$
\STATE Set $\mat{A}_{\rho_j} \leftarrow \widehat{\mat{U}}$
\STATE Update $\mat{G}_{(\rho_j)} \leftarrow \widehat{\mat{\Sigma}}\widehat{\mat{V}}^\top$
\ENDFOR
\STATE $\ten{G} \leftarrow \mat{G}_{(\rho_d)}$, in tensor form. 
\end{algorithmic}
\caption{Randomized STHOSVD}
\label{alg:rsthosvd}
\end{algorithm}

\subsection{Error Analysis}
In the results below, we assume that the matrices $\{\mat\Omega_j\}_{j=1}^d$ are standard Gaussian random matrices of appropriate sizes.
\begin{theorem}[Randomized HOSVD]
\label{thm:rhosvd_err}
Let $\tenh{X} = [\ten{G};\mat{A}_1,\dots,\mat{A}_d]$ be the output of \cref{alg:rhosvd} with inputs target rank $\mat{r} = (r_1,r_2,\dots,r_d)$ and oversampling parameter $p \geq 2$. Furthermore, assume that $p$ satisfies $r_j + p \leq \min\{I_j,\prod_{i\neq j} I_i\}$ for $j=1,\dots,d$. Then, the expected error in the approximation is 
\begin{eqnarray}
\mb{E}_{\{\mat{\Omega}_k\}_{k=1}^d} \|\ten{X}-\tenh{X} \|_F \leq & \> \left(\sum_{j=1}^d \left(1+\frac{r_j}{p-1} \right) \Delta_j^2(\ten{X}) \right)^{1/2} \\
\leq & \>  \left(d+\frac{\sum_{j=1}^d r_j}{p-1}  \right)^{1/2} \|\ten{X} - \tenh{X}_\text{opt} \|_F. \label{eqn:rhosvd_err}
\end{eqnarray}
\end{theorem}

\begin{proof}
From~\cref{lem:proj}, we can write 
$$ \mb{E}_{\{\mat{\Omega}_k\}_{k=1}^d} \| \ten{X} - \tenh{X} \|_F^2 \leq \mb{E}_{\{\mat{\Omega}_k\}_{k=1}^d} \sum_{j=1}^d \|\ten{X} \times_j (\mat{I} - \mat{A}_j\mat{A}_j^\top) \|_F^2 = \sum_{j=1}^d \mb{E}_{\mat{\Omega}_j} \| \ten{X} \times_j (\mat{I} - \mat{A}_j\mat{A}_j^\top) \|_F^2, $$ 
where the equality comes from linearity of expectations and the independence of $\mat{\Omega}_j$ for each mode $j$.  
We can unfold each term in the summation as $\| \ten{X} \times_j (\mat{I} - \mat{A}_j\mat{A}_j^\top) \|_F^2 = \| (\mat{I}-\mat{A}_j\mat{A}_j^\top) \mat{X}_{(j)} \|_F^2$. Then, by applying~\cref{thm:randsvd_err}, we can bound the expected value of the squared error in each mode to obtain
$$\mb{E}_{\{\mat{\Omega}_k\}_{k=1}^d} \| \ten{X} - \tenh{X} \|_F^2 \leq \sum_{j=1}^d \left(1+\frac{r_j}{p-1} \right) \Delta^2_j(\ten{X}).$$
Finally, H{\"o}lder's inequality gives 
$$\mb{E}_{\{\mat{\Omega}_k\}_{k=1}^d} \| \ten{X} - \tenh{X} \|_F \leq \left(\mb{E}_{\{\mat{\Omega}_k\}_{k=1}^d} \| \ten{X} - \tenh{X} \|_F^2 \right)^{1/2} \leq \left(\sum_{j=1}^d \left(1+\frac{r_j}{p-1} \right) \Delta^2_j(\ten{X}) \right)^{1/2}.$$
For the second inequality, recall that   
$\Delta^2_j(\ten{X}) \leq \|\ten{X} - \ten{\hat{X}}_\text{opt} \|_F^2$
from~\cref{eqn:xbestdelta}. Thus, combined with the previous inequality, we have
$$\mb{E}_{\{\mat{\Omega}_k\}_{k=1}^d} \| \ten{X} - \tenh{X} \|_F \leq \left(d + \frac{\sum_{j=1}^dr_j}{p-1} \right)^{1/2} \|\ten{X} - \tenh{X}_\text{opt} \|_F. $$
\end{proof}

To compare this result to the approximation error obtained using the HOSVD algorithm, we consider a few special cases.  Let $r = \max_{1 \leq j \leq d} r_j$. Then, if $p = r+1$, this factor becomes 
$$\mb{E}_{\{\mat{\Omega}_k\}_{k=1}^d}\| \ten{X} - \tenh{X} \|_F \leq \sqrt{2} \|\ten{X} - \tenh{X}_\text{HOSVD} \|_F\text{ or }\mb{E}_{\{\mat{\Omega}_k\}_{k=1}^d}\| \ten{X} - \tenh{X} \|_F \leq \sqrt{2d} \|\ten{X} - \tenh{X}_\text{opt} \|_F.$$  
Similarly, if we choose $p = \lceil \frac{r}{\epsilon} \rceil +1$ for some $\epsilon > 0$, the error satisfies $$\mb{E}_{\{\mat{\Omega}_k\}_{k=1}^d}\| \ten{X} - \tenh{X} \|_F \leq \sqrt{1+\epsilon} \|\ten{X} - \tenh{X}_\text{HOSVD} \|_F\text{ or }\mb{E}_{\{\mat{\Omega}_k\}_{k=1}^d}\| \ten{X} - \tenh{X} \|_F \leq \sqrt{d(1+\epsilon)} \|\ten{X} - \tenh{X}_\text{opt} \|_F.$$ This shows that the application of a randomized SVD in each mode of the tensor does not seriously deteriorate the accuracy compared to using an SVD.  

Now consider the randomized STHOSVD approximation. When bounding the error in expectation in this case, it is important to note that at each intermediate step, the partially truncated core tensor is a random tensor.  This is in contrast to the R-HOSVD, where we only needed to account for $\mat{\Omega}_j$ for each mode because the operations are independent across the modes. Computing the modes sequentially means, when processing mode $j$, we must account for all the random matrices $\mat{\Omega}_k$, where $k = 1,\dots,j$.  

For this theorem, we use the same notation introduced in \cref{ssec:hosvd} for the STHOSVD, in that the partially truncated core tensor at step $j$ is $\ten{G}^{(j)} = \ten{X} \bigtimes_{i=1}^j \mat{A}_i^\top$, giving a partial approximation $\tenh{X}^{(j)} = \ten{G}^{(j)} \bigtimes_{i=1}^j \mat{A}_i$.

\begin{theorem}[Randomized STHOSVD]
\label{thm:rst_err}
Let $\tenh{X} = [\ten{G};\mat{A}_1,\dots,\mat{A}_d]$ be the output of~\cref{alg:rsthosvd} with inputs target rank $\mat{r} = (r_1,r_2,\dots,r_d)$, processing order $\rho$, and oversampling parameter $p \geq 2$. Furthermore, assume that $p$ satisfies $r_j + p \leq \min\{I_j,\prod_{i\neq j} I_i\}$ for $j=1,\dots,d$. Then, the approximation error in expectation is
\begin{eqnarray}
\mb{E}_{\{\mat{\Omega}_k\}_{k=1}^d} \|\ten{X}-\ten{\widehat{X}}\|_F \leq & \> \left(\sum_{j = 1}^d \left(1+\frac{r_j}{p-1}\right)\Delta^2_j(\ten{X}) \right)^{1/2} \\
\leq & \> \left(d+\frac{\sum_{j=1}^d r_j}{p-1}  \right)^{1/2} \|\ten{X} - \ten{\widehat{X}}_\text{opt} \|_F. \label{eqn:rst_err}
\end{eqnarray}
\end{theorem}
\begin{proof}
We first assume that the processing order is $\rho = [1,\dots,d]$. The first equality in \cref{lem:proj} and the linearity of expectations together give 
\begin{equation}\label{eqn:lin_exp}
\mb{E}_{\{\mat{\Omega}_k\}_{k=1}^d} \|\ten{X}-\tenh{X} \|_F^2 = \sum_{j=1}^d  \mb{E}_{\{\mat{\Omega}_k\}_{k=1}^d} \|\tenh{X}^{(j-1)}-\tenh{X}^{(j)} \|_F^2 = \sum_{j=1}^d  \mb{E}_{\{\mat{\Omega}_k\}_{k=1}^j} \|\tenh{X}^{(j-1)}-\tenh{X}^{(j)} \|_F^2.
\end{equation}
We have used the fact that the $j$-th term in the summation does not depend on the random matrices $\{\mat\Omega_k\}_{k > j}$. We first consider $\mb{E}_{\{\mat{\Omega}_k\}_{k=1}^j} \|\tenh{X}^{(j-1)}-\tenh{X}^{(j)} \|_F^2.$  Since all the $\mat{\Omega}_k$'s are independent, we can write the expectation in an iterated form as
$$\mb{E}_{\{\mat{\Omega}_k\}_{k=1}^j} \|\tenh{X}^{(j-1)}-\tenh{X}^{(j)} \|_F^2 = \mb{E}_{\{\mat{\Omega}_k\}_{k=1}^{j-1}} \left\{\mb{E}_{\mat{\Omega}_j} \|\tenh{X}^{(j-1)}-\tenh{X}^{(j)} \|_F^2 \right\}.$$
The $j$-th term, which measures the difference in the sequential iterates, can be expressed as 
$$\| \tenh{X}^{(j-1)} - \tenh{X}^{(j)} \|_F^2 = \|\ten{G}^{(j-1)} \bigtimes_{i=1}^{j-1} \mat{A}_i \times_j (\mat{I} - \mat{A}_j \mat{A}_j^\top) \|_F^2.$$
Now let 
\[ \mat{Z}_j \equiv \underbrace{\mat{I} \otimes \cdots \otimes \mat{I}}_{d-j \text{ terms}} \otimes \mat{A}_{j-1} \otimes \cdots \otimes\mat{A}_1 .\]
If we unfold the difference $\tenh{X}^{(j-1)} - \tenh{X}^{(j)}$ along the $j$-th mode, using~\cref{eqn:kron} we have 
\begin{equation}\label{eqn:exp_unfold}
\begin{aligned}
 \|\tenh{X}^{(j-1)}-\tenh{X}^{(j)} \|_F^2 &=  \| (\mat{I} - \mat{A}_j\mat{A}_j^\top) \mat{G}^{(j-1)}_{(j)} \mat{Z}_j^\top  \|_F^2 \\
&\leq \|(\mat{I}-\mat{A}_j\mat{A}_j^\top) \mat{G}^{(j-1)}_{(j)} \|_F^2.
\end{aligned}
\end{equation}
The inequality comes as $\mat{Z}_j$ has orthonormal columns for every $j$. 

Now, let $\mat{\Gamma}_j = \mat{G}_{(j)}^{(j-1)}$ for simplicity. We take expectations and bound this last quantity in \cref{eqn:exp_unfold} using~\cref{thm:randsvd_err} (keeping $\{\mat\Omega_k\}_{k=1}^{j-1}$ fixed), as
\begin{equation}\label{eqn:randbound}
\mb{E}_{\mat{\Omega}_j} \| (\mat{I} - \mat{A}_j\mat{A}_j^\top) \mat{\Gamma}_j \|_F^2 \leq  \left(1 + \frac{r_j}{p-1} \right) \sum_{i=r_j+1}^{I_j} \sigma_i^2(\mat{\Gamma}_j).
\end{equation}
We recall the definition and properties of Loewner partial ordering \cite[Section 7.7]{horn1990matrix}. Let $\mat{M},\mat{N} \in \mathbb{R}^{n\times n}$ be symmetric; $\mat{M} \preceq \mat{N}$ means $\mat{N}-\mat{M} $ is positive semidefinite. For $\mat{S} \in \mathbb{R}^{n\times m}$, then $\mat{S}^\top \mat{MS} \preceq \mat{S}^\top\mat{NS}$. 
Furthermore, $\lambda_i(\mat{M}) \leq \lambda_i(\mat{N})$ for $i=1,\dots,n$.  Since $\mat{Z}_j$ has orthonormal columns, $\mat{Z}_j\mat{Z}_j^\top$ is a projector so that
$$\mat{\Gamma}_j\mat{\Gamma}_j^\top = \mat{X}_{(j)} \mat{Z}_j \mat{Z}_j^\top \mat{X}_{(j)}^\top \preceq \mat{X}_{(j)}\mat{X}_{(j)}^\top,$$ and the singular values of $\mat{\Gamma}_j$, which are squared eigenvalues of $\mat{\Gamma}_j\mat{\Gamma}_j^\top$, satisfy
\begin{equation}\label{eqn:B_svalues}
\sum_{i=r_j+1}^{I_j} \sigma_i^2(\mat{\Gamma}_j) \leq \sum_{i=r_j+1}^{I_j} \sigma_i^2(\mat{X}_{(j)}) = \Delta_j^2(\ten{X}).
\end{equation}
To summarize, \cref{eqn:lin_exp,eqn:exp_unfold,eqn:randbound,eqn:B_svalues} combined give
\begin{equation*}
\mb{E}_{\{\mat{\Omega}_k\}_{k=1}^d} \|\ten{X}-\tenh{X} \|_F^2 \leq \sum_{j=1}^d \mb{E}_{\{\mat{\Omega}_k\}_{k=1}^{j-1}} \left(1 + \frac{r_j}{p-1} \right) \Delta_j^2(\ten{X}) = \sum_{j=1}^d \left(1 + \frac{r_j}{p-1} \right) \Delta_j^2(\ten{X}).
\end{equation*}
The equality follows since the tensor $\ten{X}$ is deterministic. 
Finally, we have by H{\"o}lder's inequality and~\cref{eqn:delta} that
\begin{equation*}
\mb{E}_{\{\mat{\Omega}_k\}_{k=1}^d} \|\ten{X}-\tenh{X} \|_F \leq \left(\sum_{j = 1}^d \left(1+\frac{r_j}{p-1}\right) \Delta_j^2(\ten{X}) \right)^{1/2}  \leq \left(d+\frac{\sum_{j=1}^d r_j}{p-1}  \right)^{1/2} \|\ten{X} - \tenh{X}_\text{opt} \|_F.
\end{equation*}

In the general case, when the processing order does not equal $\rho = [1,\dots,d]$, the proof is similar. We only need to work with the processed order. We omit the details. 
\end{proof} 
We make several observations regarding \cref{thm:rst_err}. First, the upper bound for the error is the same for the R-STHOSVD as for the R-HOSVD (\cref{thm:rhosvd_err}). Second, this result says that the error bound for R-STHOSVD is independent of the processing order.  This means that while some processing orders may result in more accurate decompositions, every processing order has the same worst-case error bound. Our recommendation is to pick a processing order that minimizes the computational cost; see \cref{ssec:comp} for details. Third, the discussion following \cref{thm:rhosvd_err} regarding the choice of the oversampling parameter is applicable to the R-STHOSVD as well.

\paragraph{Extensions} There are several possible extensions of our results. First, we can extend this analysis to develop concentration results that give insight into the tail bounds. These can be obtained by combining our analysis with the results from, e.g., \cite[Theorem 5.8]{gu2015subspace}. Second, other distributions of random matrices may be used instead of standard Gaussian random matrices. Examples include Rademacher random matrices, sparse Rademacher random matrices, subsampled randomized Fourier transforms, etc. It is also possible to combine the analysis with the probabilistic bounds for the other decompositions. Typically, expectation bounds of the type presented in \cref{thm:rst_err} and \cref{thm:rhosvd_err} are only possible for the standard Gaussian random matrices and one can only develop tail bounds. We do not pursue these extensions here but may consider them in future work.

\subsection{Computational Cost}\label{ssec:comp}
We discuss the computational costs of the proposed randomized algorithms and compare them against the HOSVD and the STHOSVD algorithms. We make the following assumptions. First, we assume that the tensors are dense and our implementations take no advantage of their structure. Second, we assume that the target ranks in each dimension are sufficiently small, i.e., $r_j \ll I_j$ so that we can neglect the computational cost of the QR factorization and the truncation steps of the RandSVD algorithm. Third, we assume that the random matrices used in the algorithms are standard Gaussian random matrices. If other distributions are used, the computational cost may be lower. Finally, for the STHOSVD and R-STHOSVD algorithms,  we assume that the processing order is $\rho = [1,2,\dots,d].$

The computational cost of both HOSVD and STHOSVD was discussed in~\cite{vannieuwenhoven2012new}, and is reproduced in \cref{tab:cost}. In this paper, we also provide an analysis of the computational cost of R-HOSVD and R-STHOSVD, which is summarized in~\cref{tab:cost}.  The table includes the costs for both a general tensor $\ten{X} \in \mb{R}^{I_1 \times I_2 \times \dots \times I_d}$ with target rank $\mat{r} = (r_1,r_2,\dots,r_d)$, as well as for the special case when $\ten{X} \in \mb{R}^{I \times I \times \dots \times I}$ with target rank $\mat{r} = (r,r,\dots,r)$.  For ease of notation, denote the product $\prod_{k=i}^j I_k$ by $I_{i:j}$, and similarly  $\prod_{k=i}^j r_k = r_{i:j}$ for $1 \leq i \leq j \leq d$. 
The dominant costs of each algorithm lie in computing the SVD of the unfoldings (the first term in each summation) and forming the core tensor (the second term in each summation). 

\begin{table}[!ht]\centering
\caption{Computational Cost for the HOSVD, R-HOSVD, STHOSVD, and R-STHOSVD algorithms. The first term in each expression is the cost of computing an SVD of the mode unfoldings, and the second is the cost of forming the core tensor.}
\label{tab:cost}
\begin{tabular}{c | c | c}
Algorithm & Cost for $\ten{X} \in \mb{R}^{I_1 \times \dots \times I_d}$ & Cost for $\ten{X} \in \mb{R}^{I \times \dots \times I}$ \\ 
\hline 
HOSVD & $\mc{O}\left( \sum_{j=1}^d I_j I_{1:d} + \sum_{j=1}^d r_{1:j} I_{j:d} \right)$ & $\mc{O}\left(dI^{d+1}+\sum_{j=1}^d r^jI^{d-j+1}\right)$ \\[8pt]
R-HOSVD & $\mc{O}\left(\sum_{j=1}^d r_j I_{1:d} + \sum_{j=1}^d r_{1:j} I_{j:d} \right)$ & $\mc{O}\left(drI^d+\sum_{j=1}^d r^jI^{d-j+1}\right)$ \\[8pt] \hline
STHOSVD & $\mc{O}\left(\sum_{j=1}^d I_j r_{1:j-1} I_{j:d} + \sum_{j=1}^d r_{1:j} I_{j+1:d}\right)$ & $\mc{O}\left(\sum_{j=1}^d r^{j-1}I^{d-j+2}+ r^jI^{d-j}\right)$ \\[8pt]
R-STHOSVD & $\mc{O}\left(\sum_{j=1}^d r_{1:j} I_{j:d}  + \sum_{j=1}^d r_{1:j} I_{j+1:d}\right)$ & $\mc{O}\left(\sum_{j=1}^d r^{j}I^{d-j+1}+ r^jI^{d-j}\right)$
\end{tabular}
\end{table}

\paragraph{Processing order} 
In all the previous analyses for the R-STHOSVD algorithm, we took the processing order of modes to be $\rho = [1,2, \dots , d]$. The error in the approximation depends on the choice of the processing order;  however, \cref{thm:rst_err} suggests that the worst case error is independent of the processing mode. For this reason, we choose a processing order that minimizes the computational cost.  Since the dominant cost at each step $j$ is a randomized SVD with a cost of $\mc{O}(r_{\rho_1:\rho_j}I_{\rho_j:\rho_d})$, we can minimize this by choosing to process the largest modes first. That is, we process the modes in order of decreasing mode sizes. Note that this is in contrast to the approach taken by \cite{vannieuwenhoven2012new} for the STHOSVD, in that they process in the order of increasing mode sizes to minimize the cost of the standard SVD in each step.

\section{Adaptive Randomized Tensor Decompositions}\label{sec:adapt}
In the algorithms described in the previous section, we had to assume prior knowledge of the target rank. This knowledge may not be available, or may be difficult to estimate in practice. Given a tensor $\ten{X}$, it is often desirable to produce a decomposition $\tenh{X}$ that satisfies 
\[ \| \ten{X} - \tenh{X} \|_F \leq \varepsilon \|\ten{X}\|_F,\]
where $0 < \varepsilon < 1$ is a user-defined parameter. Note that there may not be a unique tensor $\tenh{X}$ that satisfies this inequality, but it is desirable to find a tensor with a small multirank that does satisfy this inequality.  We first explain the adaptive randomized algorithm to find a low-rank matrix approximation, and explain how this can be extended to the tensor case.

Several adaptive randomized range finders have been proposed in the literature~\cite{halko2011finding,gu2016efficient}. Given a matrix $\mat{X}$ and a tolerance $\varepsilon > 0$, the goal is to find a matrix $\mat{Q}$ with orthonormal columns that satisfies 
\begin{equation}\label{eqn:tolerance} \| \mat{X} - \mat{QQ}^\top \mat{X}\| \leq \varepsilon \|\mat{X}\|. \end{equation}  
The number of columns of $\mat{Q}$ is taken to be the rank of the low rank approximation. The adaptive algorithms begin with a small number of columns of the random matrix $\mat\Omega$ to estimate the range $\mat{Q}$ and then sequentially increase the number of columns of $\mat\Omega$ until the matrix $\mat{Q}$ satisfies \cref{eqn:tolerance}.  In our paper, we use a version of the adaptive randomized range finding algorithm first proposed by \cite{martinsson2016randomized} and refined in \cite{gu2016efficient}. We denote the result of this algorithm as $\mat{Q} =  \text{AdaptRangeFinder}(\mat{X},\varepsilon,b)$, where $\mat{X}$ is the matrix to be approximated, $\varepsilon$ is the requested relative error tolerance, and $b$ is a blocking integer to determine how many columns of $\mat\Omega$ to draw at a time.

\paragraph{Adaptive R-HOSVD} We now explain how the adaptive range finder can be used for computing tensor factorizations. For the R-HOSVD, we  apply this adaptive matrix algorithm to each mode unfolding $\mat{X}_{(j)}$. This gives a factor matrices $\mat{A}_j$ for each mode $j=1,\dots,d$.  Given some tolerance $\varepsilon$, the approximation error $\|\ten{X}-\tenh{X} \|_F \leq \varepsilon \|\ten{X}\|_F$ can be achieved if we choose the factor matrices $\mat{A}_j$ to satisfy
$$\| \mat{X}_{(j)} - \mat{A}_j\mat{A}_j^\top \mat{X}_{(j)} \|_F = \| \ten{X} \times_{j} (\mat{I} - \mat{A}_j\mat{A}_j^\top) \|_F \leq \varepsilon \|\ten{X} \|_F /\sqrt{d}.$$ 
Thus, we have apportioned an equal amount of error tolerance to each mode unfolding. Combined with \cref{lem:proj}, this ensures that an overall relative error $\varepsilon$ is achieved. This approach is summarized in \cref{alg:adapt_rhosvd}. Suppose we want a more flexible approach, in which a different tolerance  $\epsilon_j$ is chosen for mode $j=1,\dots,d$. This may be necessary, if we know in advance that we want to avoid compressing along some modes. In general, we may use any sequence $\epsilon_j$, so long as it satisfies $(\sum_{j=1}^d\epsilon_j^2) = \varepsilon^2$. Indeed, setting $\epsilon_j = 0$ for selected modes ensures that no compression is performed across those modes.  Note that the choice $\epsilon_j = \varepsilon/\sqrt{d}$ for $j=1,\dots,d$ automatically satisfies this equality.

\begin{algorithm}[!ht]
\begin{algorithmic}[1]
\REQUIRE $d$-mode tensor $\ten{X} \in \mb{R}^{I_1 \times I_2 \times \dots \times I_d}$, tolerance $\varepsilon \geq 0$, blocking integer $b \geq 1$
\ENSURE $\tenh{X} = [\ten{G};\mat{A}_1,\dots,\mat{A}_d]$
\FOR {$j=1:d$}
\STATE $\mat{A}_j =$ AdaptRangeFinder($\mat{X}_{(j)},\frac{\varepsilon}{\sqrt{d}},b$)
\ENDFOR
\STATE Form $\ten{G} = \ten{X} \bigtimes_{j=1}^d \mat{A}_j^\top$
\end{algorithmic}
\caption{Adaptive R-HOSVD}
\label{alg:adapt_rhosvd}
\end{algorithm}

\paragraph{Adaptive R-STHOSVD} The same approach can be extended to the R-STHOSVD algorithm. We define the intermediate tensors $\ten{X}^{(j)} = \ten{X}\bigtimes_{i=1}^j\mat{A}_i\mat{A}_i^\top$ for $j=1,\dots,d$ and $\ten{X}^{(0)} = \ten{X}$. Analogously, we define the intermediate core tensor $\ten{G}^{(j)} = \ten{X}\bigtimes_{i=1}^j\mat{A}_i^\top$ with $\ten{G}^{(0)} = \ten{G}$. Furthermore, we choose the processing order $\rho = [1,2,\dots,d]$ for simplicity. In this case, we choose the factor matrices $\mat{A}_j$ in order to ensure that the successive iterates satisfy 
\[ \|\ten{X}^{(j-1)} - \ten{X}^{(j)}\|_F = \|\ten{G}^{(j-1)} \bigtimes_{i=1}^{j-1} \mat{A}_j \times_j(\mat{I}-\mat{A}_j\mat{A}_j^\top)\|_F \leq \frac{\varepsilon }{\sqrt{d}}\|\ten{X}\|_F, \quad j=1,\dots,d.\]
 Applying the first part of \cref{lem:proj}, we obtain 
\[ \|\ten{X} - \ten{X}^{(d)}\|_F^2 = \sum_{j=1}^d\|\ten{X}^{(j-1)} - \ten{X}^{(j)}\|_F^2 \leq \varepsilon^2 \|\ten{X}\|_F^2.\]
The details are provided in \cref{alg:adapt_sthosvd}, which uses a general processing order. As with R-HOSVD, we may elect to use the same error tolerance $\varepsilon/\sqrt{d}$ for each mode unfolding, or use a different tolerance $\epsilon_j$ for iterate $j=1,\dots,d$.

\begin{algorithm}[!ht]
\begin{algorithmic}[1]
\REQUIRE $d$-mode tensor $\ten{X} \in \mb{R}^{I_1 \times I_2 \times \dots \times I_d}$, processing order $\rho$, tolerance $\varepsilon \geq 0$, blocking integer $b \geq 1$
\ENSURE $\tenh{X} = [\ten{G};\mat{A}_1,\dots,\mat{A}_d]$
\STATE Set $\ten{G} \leftarrow \ten{X}$
\FOR {$j=1:d$}
\STATE $\mat{A}_{\rho_j} = \text{AdaptRangeFinder}(\mat{G}_{(\rho_j)},\frac{\varepsilon}{\sqrt{d}},b)$
%\STATE Set $\mat{A}_{\rho_j} \leftarrow \mat{Q}$
\STATE Update $\ten{G}_{(\rho_j)} \leftarrow \mat{A}_{\rho_j}^\top \mat{G}_{(\rho_j)}$
\ENDFOR
\STATE $\ten{G} \leftarrow \mat{G}_{(\rho_d)}$, in tensor format
\end{algorithmic}
\caption{Adaptive R-STHOSVD}
\label{alg:adapt_sthosvd}
\end{algorithm}

\section{Structure-preserving decompositions} \label{sec:sparse}
In this section, we are interested in computing a low-rank decomposition in the Tucker format in which the core tensor $\ten{G} \in \mb{R}^{r_1\times \dots \times r_d}$ has entries that are explicitly taken from the original tensor $\ten{X} \in \mb{R}^{I_1\times \dots \times I_d}$. That is, $\ten{X} \approx \ten{G} \bigtimes_{j=1}^d \mat{A}_j$
where $\mat{A}_j \in \mb{R}^{I_j \times r_j}$ with $j=1,\dots,d$ are the factor matrices (that do not necessarily have orthonormal columns). We call such a decomposition {\em structure-preserving} since the core  tensor retains favorable properties (e.g., sparsity, nonnegativity, binary or integer counts) of  the original tensor. This generalizes a related decomposition proposed for matrices in \cite{cheng2005compression}. Related to the structure-preserving decompositions, prior work includes a higher-order interpolatory decomposition \cite{saibaba2017hoid,drineas2007randomized,mahoney2008tensor} of the form 
\[ \ten{X} \approx \ten{G} \bigtimes_{j=1}^d \mat{C}_j  \qquad \ten{G} = \ten{X} \bigtimes_{j=1}^d \mat{C}_j^\dagger, \]
where the matrices $\{\mat{C}_j\}_{j=1}^d$ have entries from the original tensor (specifically, columns selected from the appropriate mode-unfoldings), and ${}^\dagger$ represents the Moore-Penrose pseudoinverse.

\subsection{Algorithm}
We first explain our algorithm for the matrix $\mat{X}$. We compute a basis $\mat{Q}$ using the randomized range finding algorithm. Instead of computing the low-rank approximation $\mat{QQ}^\top \mat{X}$,  as was done in \cref{alg:randsvd}, we first identify a set of well-conditioned rows of $\mat{Q}$. This is implemented as a  selection operator denoted by the matrix $\mat{P}$,  which contains columns from the identity matrix. We then use the low-rank representation $$\mat{X} \approx \mat{Q}(\mat{P}^\top\mat{Q})^{-1}\mat{P}^\top \mat{X} = \mat{A}\widehat{\mat{X}},$$
where $\mat{A} = \mat{Q}(\mat{P}^\top\mat{Q})^{-1}$ and $\widehat{\mat{X}} = \mat{P}^\top \mat{X}.$ 
We see that the matrix $\mat{A}$ does not have orthonormal columns, but is well-conditioned, and that the matrix $\widehat{\mat{X}}$ contains rows from the matrix $\mat{X}$ as determined by the selection operator $\mat{P}$. The selection operator can be determined in a variety of ways: using the discrete empirical interpolation method~\cite{sorensen2016deim}, pivoted QR factorization~\cite{drmac2016new}, or strong rank-revealing QR (sRRQR) factorization~\cite{gu1996efficient,drmac2018discrete} (which is used in this paper).

The idea behind our algorithms is the following: at each step, given the core tensor $\ten{G}$ we first unfold this tensor and use a randomized range finder on the unfolding to obtain a basis $\mat{Q}_j$. Then, we use the sRRQR algorithm to determine the selection operator that identifies well-conditioned rows of $\mat{Q}_j$ to determine the core tensor for the next step. The details of the algorithm are provided in \cref{alg:sparse}. A few things are worth noting. First, the core tensor at each step contains elements from the original tensors, so that the final core tensor has entries from the original tensor; this makes the algorithm structure preserving. Second, in contrast to \cref{alg:rhosvd,alg:rsthosvd}, the factor matrices do not have orthonormal columns; if orthonormal columns are desired, a postprocessing step can be performed (a thin-QR factorization of each factor matrix to obtain the basis, followed by an aggregation step in which the core tensor is multiplied with all the triangular factors). Third, the resulting tensor is of rank-$(r_1+p,\dots,r_d+p)$, which is also in contrast to other algorithms that produce decompositions of the rank $(r_1,\dots,r_d)$. Once again, these factors can be recompressed using a postprocessing step; see, for example,~\cite{kressner2017recompression}.

\begin{algorithm}[!ht]
\begin{algorithmic}[1]
\REQUIRE $d$-mode tensor $\ten{X} \in \mb{R}^{I_1 \times I_2 \times \dots \times I_d}$, target rank vector $\mat{r} \in \mb{N}^d$, \\
$\quad$ Oversampling parameter $p\geq 0$  such that $r_j+p < \min \{I_j,\prod_{i\neq j} I_i\}$ for $j=1,\dots,d$, processing order $\rho$
\ENSURE $\tenh{X} = [\ten{G};\mat{A}_1,\dots,\mat{A}_d]$
\STATE Set $\ten{G} = \ten{X}$
\FOR{$j=1:d$}
\STATE Draw Gaussian matrix $\mat{\Omega}_{\rho_j} \in \mb{R}^{\prod_{\rho_i \neq \rho_j} I_{\rho_i} \times (r_{\rho_j}+p)}$
\STATE Form $\mat{Y} \leftarrow \mat{G}_{(\rho_j)} \mat{\Omega}_{\rho_j}$
\STATE Thin QR factorization $\mat{Y} = \mat{Q}_{\rho_j}\mat{R}$
\STATE Use strong RRQR on $\mat{Q}_{\rho_j}^\top$ with parameter $\eta = 2$ 
\[ \mat{Q}_{\rho_j}^\top \bmat{\mat{S}_1 & \mat{S}_2} = \mat{Z} \bmat{ \mat{R}_{11} & \mat{R}_{12}}. \]  
Let $\mat{P}_{\rho_j}  = \mat{S}_1 \in \mb{R}^{I_{\rho_j} \times r_{\rho_j}}$ which contains the columns from the identity matrix. 
\STATE Form $\mat{A}_{\rho_j} = \mat{Q}_{\rho_j}(\mat{P}_{\rho_j}^\top \mat{Q}_{\rho_j})^{-1}$
\STATE Update $\mat{G}_{(\rho_j)} \leftarrow \mat{P}_{\rho_j}^\top \mat{G}_{(\rho_j)}$
\ENDFOR
\STATE Set $\ten{G} = \mat{G}_{(\rho_d)}$, in tensor format
\end{algorithmic}
\caption{Structure-preserving STHOSVD (SP-STHOSVD)}
\label{alg:sparse}
\end{algorithm}

\cref{alg:sparse} is particularly beneficial for sparse tensors. Although sparse tensors can be efficiently stored in an appropriate tensor format (e.g.,~\cite{bader2007efficient}), a straightforward application of either \cref{alg:rhosvd} or \cref{alg:rsthosvd} produces dense intermediate tensors that may be prohibitively expensive to store, even though the final decomposition may be economical in terms of storage costs. On the other hand, in \cref{alg:sparse}, each intermediate core tensor is sparse, and only contains entries from the original tensor. Therefore, the intermediary core tensors can be efficiently stored in the same sparse tensor format. Besides the savings in memory, storing in sparse tensor format is efficient for computational reasons since tensor and matrix products are cheaper to compute.

\paragraph{Computational Cost} For simplicity, we assume a processing order of $\rho = [1,2,\dots, d]$. The two dominant costs of \cref{alg:sparse} for each mode are obtaining the basis $\mat{Q}_j$ and computing an sRRQR to determine $\mat{P}_j$. Let $\mathsf{nnz}(\ten{G}^{(j)})$ denote the number of nonzeros in the core tensor at step $j$. Letting  $\ell_j = r_j+p$, the cost of forming the product of the (unfolded) core tensor with a random matrix $\mat{\Omega}_j$, over all $d$ modes is $\mc{O}(\sum_{j=1}^d \mathsf{nnz}(\ten{G}^{(j)})\ell_j)$.  Computing an sRRQR of an $I_j \times \ell_j$ matrix with parameter $\eta = 2$ (parameter was called $f$ in ) costs $\mc{O}(I_j \ell_j^2)$ per mode. Combining both the dominant costs gives a total cost of $\mc{O}\left(\sum_{j=1}^d \mathsf{nnz}(\ten{G}^{(j)})\ell_j +\sum_{j=1}^d I_j \ell_j^2\right).$
This analysis shows that the computational cost of SP-STHOSVD is significantly smaller than the other algorithms presented thus far, particularly with a sparse tensor. Even when the original tensor is dense, there is a savings in computational cost compared to STHOSVD since a full SVD is not computed, and compared to R-STHOSVD since the core tensor $\ten{G}^{(j)}$ is only multiplied once per iteration. For computational reasons, we still use the processing order in \cref{ssec:comp}.

\subsection{Error Analysis}

We now present the error analysis for \cref{alg:sparse}. There are two major difficulties here in extending the proofs of \cref{lem:mode_err,lem:st_mode_err}. First, we have to work with an oblique projector $\mat{\Pi}_j = \mat{Q}_j(\mat{P}_j^\top\mat{Q}_j)^{-1}\mat{P}_j^\top$, whereas in the previous analysis we used an orthogonal projector. As a consequence, we can no longer use the Pythagoras theorem to obtain~\cref{lem:proj}; instead we have to employ the triangle inequality, resulting in a weaker bound. Second, we have to work with the factor matrices $\mat{A}_j$ which no longer have orthonormal columns. We are able to show the following error bound.
\begin{theorem}
Let $\tenh{X} = [\ten{G}; \mat{A}_1,\dots,\mat{A}_d]$ be the output of~\cref{alg:sparse} with inputs target rank $\mat{r} = (r_1,\dots,r_d)$ and oversampling parameter $p \geq 2$ such that $p$ satisfies $r_j + p < \min\{I_j,\prod_{i\neq j} I_i\}$ for $j=1,\dots,d$.
 Furthermore, assume a processing order of $\rho = [1,2,\dots,d]$, and let $\eta = 2$ be the sRRQR parameter. Then, the expected approximation error is 
\vspace{-.4cm}
\begin{equation*}
\begin{aligned}
\mb{E}_{\{\mat{\Omega}_k\}_{k=1}^d} \|\ten{X}-\tenh{X} \|_F \leq & \> \sum_{j=1}^d \left( \prod_{k=1}^j g(I_k,\ell_k) \right) f_p(r_j) \Delta_j(\ten{X}) \\
\leq & \> \sum_{j=1}^d \left( \prod_{k=1}^j g(I_k,\ell_k) \right) f_p(r_j)  \| \ten{X} - \tenh{X}_\text{opt} \|_F.
\end{aligned}
\end{equation*}
where $g(I,r) = \sqrt{1+4r(I -r)}$, and $f_p(r) = \sqrt{1+\frac{r}{p-1}}$. Furthermore, the matrices $\{\mat{A}_j\}_{j=1}^d$ each contain an $\ell_j \times \ell_j$ identity matrix and 
$$1 \leq  \|\mat{A}_j\|_2 \leq g(I_j, \ell_j) \qquad j=1,\dots,d.$$
\label{thm:pass_error}
\end{theorem}
\begin{proof}
First, consider the factor matrices $\mat{A}_j = \mat{Q}_j (\mat{P}_j^\top \mat{Q}_j)^{-1}$ for $j=1,\dots,d$. Since $\mat{P}_j$ contains columns from the identity matrix, it is easy to verify that $\mat{A}_j$ contains the identity matrix as its submatrix. The lower bound on $\|\mat{A}_j\|_2$ follows immediately from this fact. For the upper bound, since $\mat{Q}_j$ has orthonormal columns almost surely,  
\begin{equation}\label{eqn:aj} \| \mat{A}_j\|_2 = \|(\mat{P}_j^\top \mat{Q}_j)^{-1} \|_2 \leq g(I_j,\ell_j).\end{equation}
The last step follows from \cite[Lemma 2.1]{drmac2018discrete}, where $\eta = 2$ is used as the tuning parameter for the sRRQR algorithm. 

Next, let $\ten{G}^{(j)} = \ten{X} \bigtimes_{i=1}^j \mat{P}_i^\top$ denote the partially truncated core tensor after the $j$-th step of~\cref{alg:sparse}.  Also let $\tenh{X}^{(j)} = \ten{G}^{(j)} \bigtimes_{i=1}^j\mat{A}_i$ be the $j$-th partial approximation of $\ten{X}$. Then, by the triangle inequality and the linearity of expectations, we have
\begin{equation} \label{eqn:sparse_sum}
\mb{E}_{\{\mat{\Omega}_k\}_{k=1}^d} \|\ten{X}-\tenh{X} \|_F \leq  \sum_{j=1}^d \mb{E}_{\{\mat{\Omega}_k\}_{k=1}^d}  \|\tenh{X}^{(j-1)} - \tenh{X}^{(j)} \|_F .
\vspace{-.1cm}
\end{equation}
Consider the term $\|\ten{\hat{X}}^{(j-1)}-\ten{\hat{X}}^{(j)} \|_F.$ We can simplify $\ten{\hat{X}}^{(j)}$ as 
\[\begin{aligned}
\ten{\hat{X}}^{(j)} = & \> \ten{G}^{(j)} \bigtimes_{i=1}^j \mat{A}_i = \left(\ten{G}^{(j-1)} \times_j\mat{P}_j^\top \right) \bigtimes_{i=1}^j \mat{A}_i\\
= & \> \ten{G}^{(j-1)}  \bigtimes_{i=1}^{j-1} \mat{A}_i \times_j\mat{A}_j\mat{P}_j^\top  = \ten{G}^{(j-1)}  \bigtimes_{i=1}^{j-1} \mat{A}_i \times_j\mat{\Pi}_j.
\end{aligned}  \]
Therefore, $\ten{\hat{X}}^{(j-1)}-\ten{\hat{X}}^{(j)} = \ten{G}^{(j-1)}  \bigtimes_{i=1}^{j-1} \mat{A}_i \times_j (\mat{I} - \mat\Pi_j)$. Repeated use of the submultiplicativity $\|\mat{MN}\|_F \leq \|\mat{M}\|_F \|\mat{N}\|_2$,
\begin{equation}\label{eqn:submult}
\vspace{-.1cm}
\|\ten{\hat{X}}^{(j)}-\ten{\hat{X}}^{(j-1)} \|_F \leq \> \|\ten{G}^{(j-1)} \times_j(\mat{I}-\mat{\Pi}_j) \|_F \prod_{i=1}^{j-1} \| \mat{A}_i \|_2. 
\end{equation}
 Now, observe that $\mat{\Pi}_j\mat{Q}_j\mat{Q}_j^\top = \mat{Q}_j\mat{Q}_j^\top$, implying that $\mat{I}-\mat{\Pi}_j = (\mat{I}-\mat{\Pi}_j)(\mat{I}-\mat{Q}_j\mat{Q}_j^\top)$.  Therefore, once again using submultiplicativity
 \begin{equation}\label{eqn:inter}
     \begin{aligned}
         \|\ten{G}^{(j-1)} \times_j(\mat{I}-\mat{\Pi}_j) \|_F =  & \>\|\ten{G}^{(j-1)} \times_j(\mat{I}-\mat{\Pi}_j) (\mat{I}-\mat{Q}_j\mat{Q}_j^\top) \|_F \\
         \leq &\>  \|\ten{G}^{(j-1)} \times_j (\mat{I}-\mat{Q}_j\mat{Q}_j^\top) \|_F\|\mat{I}-\mat{\Pi}_j\|_2.
     \end{aligned}
 \end{equation}

 We note that $\mat\Pi_j \neq \mat{I}$ since $\rank(\mat\Pi_j) \leq \rank(\mat{Q}_j) = r_j + p < I_j$, and $\mat{\Pi}_j \neq \mat{0}$ since $\mat{Q}_j$ has orthonormal columns (almost surely) and $\mat{P}_j^\top \mat{Q}_j$ is invertible. Therefore, using~\cite{szyld2006many} $\|\mat{I} - \mat\Pi_j\|_2 = \|\mat{\Pi}_j\|_2$. Once again using \cite[Lemma 2.1]{drmac2018discrete}
$$\|\mat{I} - \mat\Pi_j\|_2 = \|\mat{\Pi}_j\|_2 = \| (\mat{P}_j^\top\mat{Q}_j)^{-1} \|_2 \leq g(I_j,\ell_j). $$
Combining this inequality with \cref{eqn:aj,eqn:inter,eqn:submult}, we have
\[ \|\tenh{X}^{(j)} - \tenh{X}^{(j-1)} \|_F \leq  
 \|\ten{G}^{(j-1)} \times_j (\mat{I}-\mat{Q}_j\mat{Q}_j^\top) \|_F \prod_{i=1}^{j} g(I_i,\ell_i). \]
Taking expectations, and using the independence of the random matrices, we obtain 
\[ \begin{aligned}
\mb{E}_{\{\mat{\Omega}_k\}_{k=1}^d}  \|\tenh{X}^{(j)} - \tenh{X}^{(j-1)} \|_F = & \> \mb{E}_{\{\mat{\Omega}_k\}_{k=1}^{j-1}} \mb{E}_{\mat{\Omega}_j}\|\ten{G}^{(j-1)} \times_j (\mat{I}-\mat{Q}_j\mat{Q}_j^\top) \|_F \prod_{i=1}^{j} g(I_i,\ell_i) \\
= & \> \mb{E}_{\{\mat{\Omega}_k\}_{k=1}^{j-1}} \mb{E}_{\mat{\Omega}_j} \|(\mat{I}-\mat{Q}_j\mat{Q}_j^\top)\mat{G}^{(j-1)}_{(j)}\|_F \prod_{i=1}^{j} g(I_i,\ell_i) \\
\leq  & \>\left(\prod_{i=1}^j g (I_i,\ell_i) \right) f_p(r_j) \mb{E}_{\{\mat{\Omega}_k\}_{k=1}^{j-1}}\left(\sum_{i= r_j+1}^{I_j} \sigma_i^2(\mat{G}_{(j)}^{(j-1)})\right)^{1/2}.
\end{aligned}
\]
In the last step, we have used \cref{eqn:nonsquared_err} and kept the random matrices $\{\mat{\Omega}_k\}_{i=1}^{j-1}$ fixed. By construction $\mat{G}_{(j)}^{(j-1)}$ is a submatrix of the mode-unfolding $\mat{X}_{(j)}$. Arguing as in the proof of \cref{thm:rst_err}, we can show  $ \sum_{i= r_j+1}^{I_j} \sigma_i^2(\mat{G}_{(j)}^{(j-1)}) \leq \Delta_j^2(\ten{X})$, for $j=1,\dots,d$
and, therefore, 
\[ \mb{E}_{\{\mat{\Omega}_k\}_{k=1}^d}  \|\tenh{X}^{(j)} - \tenh{X}^{(j-1)} \|_F \leq\left(\prod_{i=1}^j g (I_i,\ell_i) \right) f_p(r_j) \mb{E}_{\{\mat{\Omega}_k\}_{k=1}^{j-1}}\Delta_j(\ten{X}). \]
Plugging this into \cref{eqn:sparse_sum}, and using \cref{eqn:xbestdelta} we get the desired bound.

\end{proof}

In this result, we have assumed the standard processing order $\rho = [1,\dots,d]$. The analysis can be extended to other processing orders, but we omit a detailed statement here. However, note that the upper bound derived here is dependent on the processing order, which is in contrast to the bound in \cref{thm:rst_err}. Although the error bound in \cref{thm:pass_error} can be much higher than \cref{thm:rst_err}, numerical results suggest that the bounds are somewhat pessimistic and, in practice, \cref{alg:sparse} produces accurate decompositions. {Two other features are worth mentioning. First, the core tensor $\ten{G}$ contains $\prod_{j=1}^d\ell_j$ entries from the original tensor $\ten{X}$.}  Second, while each factor matrix may not be orthonormal, they are close to orthonormal and, in fact, explicitly contain the identity matrix as a submatrix.  

\subsection{Variants} One can easily develop several variations of the~\cref{alg:sparse} that may be beneficial in applications.
\begin{description}
\item [1. SP-HOSVD.] In~\cref{alg:sparse}, we handled all the modes sequentially; they can however be handled independently. For each mode $j=1,\dots,d$, we obtain a basis $\mat{Q}_j$ and a selection operator $\mat{P}_j$. The resulting decomposition is of the form$\ten{X} \approx \tenh{X} = \ten{G} \bigtimes_{j=1}^d \mat{A}_j$ where the core tensor  $\ten{G} = \ten{X} \bigtimes_{j=1}^d \mat{P}_j^\top$ and the factor matrices are of the form $\mat{A}_j = \mat{Q}_j(\mat{P}_j^\top \mat{Q}_j)^{-1}$ for $j=1,\dots,d$. The error analysis is similar to \cref{thm:pass_error} and we found that it has the same upper bound.
\item [2. Range finding.] We used a basic version of the randomized range finding algorithm to obtain the matrices $\mat{Q}_j$. Other variations are certainly possible; for example, the adaptive range finding algorithm~\cref{sec:adapt}, randomized subspace iteration~\cite[Algorithm 4.3]{halko2011finding}, or other deterministic or randomized rank-revealing decompositions. 
\item [3. Subset selection.] In \cref{alg:sparse}, we used the sRRQR algorithm for the subset selection step. In practice, this is computationally expensive, and an alternative is to use the Column-Pivoted QR factorization. This algorithm has lower computational cost but is known to fail for certain adversarial cases~\cite{gu1996efficient}. Besides deterministic algorithms for subset selection, there are several randomized techniques available, such as uniform sampling and leverage score sampling that can be used instead of sRRQR. 
\end{description}

\section{Numerical Results} \label{sec:num_results}
In this section, we study  the accuracy and the computational cost of our algorithms on several synthetic and real-world tensors. The main result we wish to verify here is that randomizing the preexisting compression algorithms decreases the computational time while not significantly increasing the error. All results were run on a desktop with a $3.4$ GHz Intel Core i7 processing unit and 16GB memory. We used two tensor packages in \textsc{matlab}, namely Tensor Toolbox \cite{TTB_Sparse} for handling sparse tensors, and Tensorlab \cite{vervliet2016tensorlab} for everything else.

\subsection{Test Problems} 
We briefly describe the four different sources of tensors that we use to validate our algorithms.
\paragraph{1. Hilbert Tensor}
Our first test tensor is a synthetic, super-symmetric tensor (invariant under the permutation of indices), where each entry is defined as 
\begin{equation}\label{eqn:hilbert}
\ten{X}_{i_1i_2 \dots i_d} = \frac{1}{i_1+i_2+\dots + i_d} \qquad 1 \leq i_j \leq I_j, \>j = 1,\dots,d. 
\end{equation}
We call this the {\em Hilbert tensor}, which generalizes the Hilbert matrix ($d=2$).  When $d=5$ and each $I_j = 25$, this tensor has $25^5 = 9,765,625$ nonzero entries. The structure of this tensor is such that the singular values of each mode-unfolding decay rapidly, which indicates that the randomized algorithms proposed in this paper are likely to be accurate.

\paragraph{2. Synthetic Sparse tensor}
For our second example, we construct a three-dimensional sparse tensor $\ten{X} \in \mb{R}^{200 \times 200 \times 200}$ as the sum of outer products as
\begin{equation}\label{eqn:sparse}
    \ten{X} = \sum_{i=1}^{10} \frac{\gamma}{i^2} \, \mat{x}_i \circ \mat{y}_i \circ \mat{z}_i + \sum_{i=11}^{200} \frac{1}{i^2} \, \mat{x}_i \circ \mat{y}_i \circ \mat{z}_i,
\end{equation}
where $\mat{x}_i, \mat{y}_i, \mat{z}_i \in \mb{R}^n$ are sparse vectors for all $i$, and $ \circ $ denotes the outer product. The sparse vectors are all generated using the \verb sprand  command with $5\%$ nonzeros. This results in the tensor $\ten{X}$ having $185,211$ nonzeros total. Furthermore, $\gamma$ is a user-defined parameter which determines the strength of the gap between the first ten terms and the last terms.
\paragraph{3. Olivetti Dataset}
The classification of facial images, or ``tensorfaces'' as popularized by \cite{vasilescu2002multilinear}, has two main steps.  The first is the compression phase, where a higher order SVD is applied to the tensor for decomposition.  The second step is the classification process, in which the decomposed tensor is used to classify new images.  We focus on the first step to efficiently  decompose the tensor of images from the Olivetti dataset \cite{olivettidataset} using the proposed randomized algorithms. This dataset contains $400$ images ($64 \times 64$ pixels) of $40$ people in $10$ different poses.  This set of images can be expressed as a three dimensional tensor $\ten{X} \in \mb{R}^{40 \times 4096 \times 10}$, in which the three modes represent people, pixels, and poses, respectively.

\paragraph{4. FROSTT database}
Our final test problems come from the formidable repository of sparse tensors and tools (FROSTT) database \cite{frosttdataset}. From this database, we choose two representative large, sparse tensors whose features are summarized in \cref{tab:frostt}.
\begin{table}[!ht]
    \centering
    \begin{tabular}{c|c|c|c}
        Original Tensor & Order & Size & Nonzeros  \\
        \hline
        NELL-2  & 3 & $12092 \times 9184 \times 28818$ & $76,879,419$   \\
        Enron  & 4 & $6066 \times 5699 \times 244268 \times 1176$ & $54,202,099$ \\
    \end{tabular}
    \vspace{.5cm}
    
    \begin{tabular}{c|c|c|c}
        Condensed Tensor & Order & Size & Nonzeros  \\
        \hline
        NELL-2  & 3 & $807 \times 613 \times 1922$ & $19,841$   \\
        Enron  & 3 & $405 \times 380 \times 9771$ & $6,131$ \\
    \end{tabular}
    \caption{Summary of sparse tensor examples from the FROSTT database---we include the details for both the full datasets and the condensed datasets used in our experiments.}
    \label{tab:frostt}
\end{table}
The NELL-2 dataset \cite{carlson2010toward} is a portion of the Never Ending Language Learning knowledge base from the ``Read the Web'' project at Carnegie Mellon University. NELL is a machine learning system that relates different entities, creating a three-dimensional dataset whose modes represent entity, relation, and entity.  The Enron dataset \cite{shetty2004enron} contains word counts in emails released during an investigation by the Federal Energy Regulatory Commission.  Here, the modes represent sender, receiver, word, and date, respectively.  

Although our implementation of SP-STHOSVD is capable of handling both full tensors in \cref{tab:frostt}, they are unfortunately too large to compute the approximation error. Then, in order to compute this error, we subsample the tensors first.  This also allows us to compare the performance of our SP-STHOSVD algorithm with others that could not handle the size of the full datasets. For the NELL-2 dataset \cite{carlson2010toward}, we subsample every 15 elements to obtain a tensor $\ten{X} \in \mb{R}^{807 \times 613 \times 1922}$. For the Enron dataset \cite{shetty2004enron}, we also need to subsample, but even with subsampling, the tensor is still too sparse to accurately compute the approximation error. To address this, we also condense the dataset to three dimensions by summing over the fourth mode.  Then to subsample, we take every 15 elements from the first two modes, and every 25 from the third.  This results in a tensor $\ten{X} \in \mb{R}^{405 \times 380 \times 9771}$.

\subsection{Numerical Experiments}
We now describe the experiments performed on the test tensors introduced in the previous subsection. 

\subsubsection{Fixed rank} Our first experiment compares the accuracy of the HOSVD and STHOSVD algorithms with their randomized counterparts, R-HOSVD and R-STHOSVD (\cref{alg:rhosvd,alg:rsthosvd}).  As inputs, we take $\ten{X}$ as defined in \cref{eqn:hilbert} with $d=5$ modes and $I_j = 25$ for $j=1,\dots,d$.  For each algorithm, we use the target rank $(r,r,r,r,r)$,  where $r$ varies from $1$ to $25$, and the same oversampling parameter $p=5$ was used in every mode.  Since this is a super-symmetric tensor, the processing order of modes does not affect the results, so we take the processing order $\rho = [1,2,3,4,5]$.  The relative error is plotted in \cref{fig:hilbert_relerr}, where we can see that the approximation error of all four algorithms is very similar and that the randomized algorithms are highly accurate. 
\begin{figure}[!ht] \centering
\includegraphics[width=.4\textwidth]{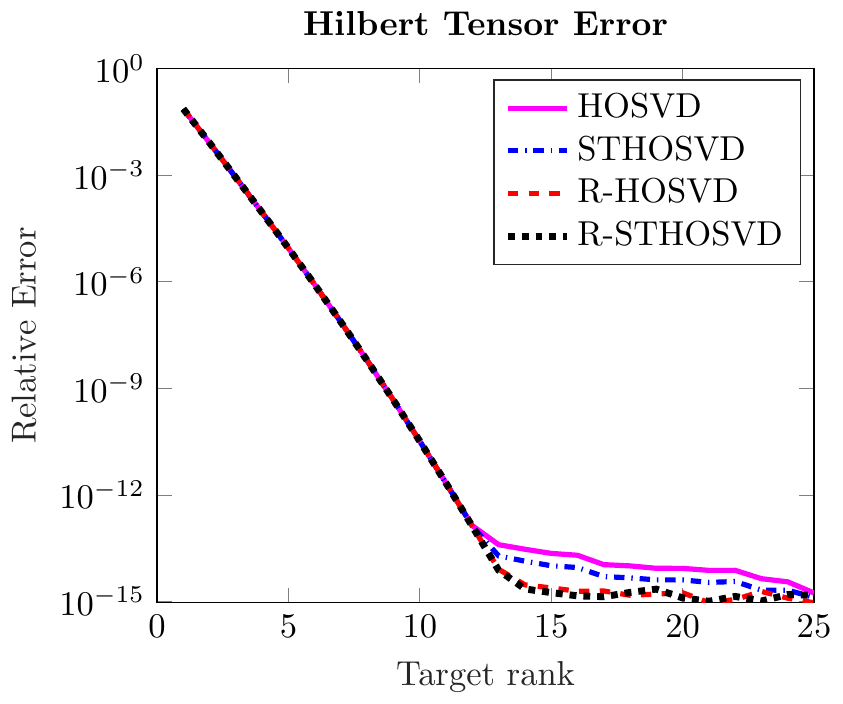}
\includegraphics[width=.4\textwidth]{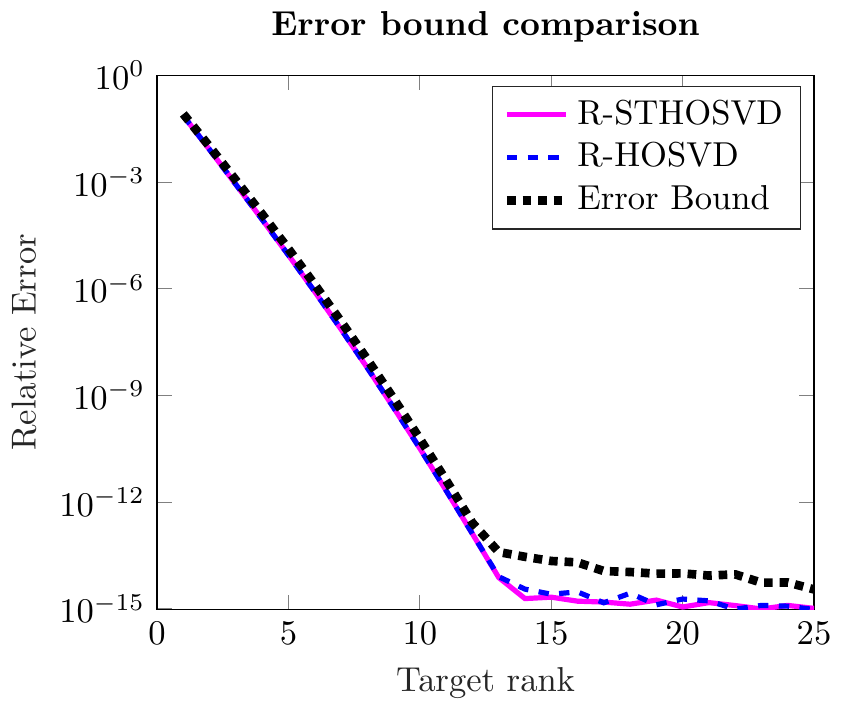}
\caption{{\bf Left:} Relative approximation error for $5$-mode Hilbert tensor $\ten{X} \in \mb{R}^{25 \times 25 \times 25 \times 25 \times 25}$ defined in \cref{eqn:hilbert}, with target rank $(r,r,r,r,r)$ and oversampling parameter $p=5$. {\bf Right:} Actual relative error for $\ten{X}$ from the R-HOSVD and R-STHOSVD algorithms compared to the calculated error bound as the target rank $(r,r,r,r,r)$ increases. Both algorithms use oversampling parameter $p=5$, and R-STHOSVD uses the processing order $\rho = [1,2,3,4,5]$.}
\label{fig:hilbert_relerr}
\end{figure}
The comparison of the cost in \cref{ssec:comp} implies that the proposed randomized algorithms are less expensive. To illustrate this, we report the runtime of the HOSVD, R-HOSVD, STHOSVD, and R-STHOSVD algorithms on $\ten{X}$ as the size of each dimension increases. For inputs, we fixed the target rank to be $(5,5,5,5,5)$, the oversampling parameter as $p=5$, and we used processing order $\rho = [1,2,3,4,5]$ in the sequential algorithms. The runtime in seconds, averaged over three runs, is shown in \cref{tab:runtime}. The theory implies that the randomized algorithms should be a factor of $n/(r+p) = 2.5$ faster than the non-randomized algorithms.  This is evident in our results.  Also, the sequential algorithms are significantly faster than the HOSVD/R-HOSVD algorithms.
\begin{table}[!ht] \centering
\begin{tabular}{c  c | c | c | c | c }
\multicolumn{1}{c}{}  &   \multicolumn{5}{c}{Algorithm}\\
\multicolumn{1}{c}{}  &   & HOSVD & R-HOSVD & STHOSVD & R-STHOSVD  \\ \cline{2-6}
& 25 & $3.0030$ & $1.2609$ & $0.6455$ & $0.2875$   \\ 
$I_j $ & 35 & $14.5255$ & $5.5288$ & $3.1974$ & $1.2090$  \\ 
$1\leq j \leq d$& 45 & $70.0536$ & $17.3744$ & $15.0629$ & $3.5587$ \\
& 50 & $101.4981$ & $27.7725$ & $21.9745$ & $5.5606$
\end{tabular}
\caption{Runtime in seconds of the HOSVD, R-HOSVD, STHOSVD, and R-STHOSVD algorithms on  the Hilbert tensor $\ten{X}$ with $d=5$, averaged over three runs. Each algorithm is run with target rank $(5,5,5,5,5)$, and the randomized algorithms use oversampling parameter $p=5$.  The STHOSVD and R-STHOSVD algorithms use the processing order $\rho = [1,2,3,4,5]$.}
\label{tab:runtime}
\end{table}

We now compare the computed results to the theoretical bounds shown in~\cref{eqn:rhosvd_err,eqn:rst_err} for the R-HOSVD and R-STHOSVD algorithms.  Since the upper bound for both algorithms is the same, we display this bound only once.  Running both algorithms with target rank $(r,r,r,r,r)$ as $r$ increases, oversampling parameter $p=5$, and processing order $\rho = [1,2,3,4,5]$ for R-STHOSVD, we can see in the right-hand plot of~\cref{fig:hilbert_relerr} that they are similar to each other and the computed bound. This closeness shows that the theoretical bounds capture the actual error well.

Next, we apply the randomized algorithms to the Olivetti dataset. The results are described in \cref{fig:faces_relerr}. As all the modes have different dimensions, we only compress the largest (the pixels) to start. The first plot shows the standard algorithms, but there is a noticeable difference in the error for the randomized and standard algorithms.  This can be explained by the fact that the singular values of the data do not decay sufficiently quickly. To fix this, we add one step of subspace iteration (see \cite{halko2011finding} for details) to the randomized SVD, giving the second plot.  Here the difference is almost nonexistent.
\begin{figure}[!ht]\centering
    \includegraphics[scale=.9]{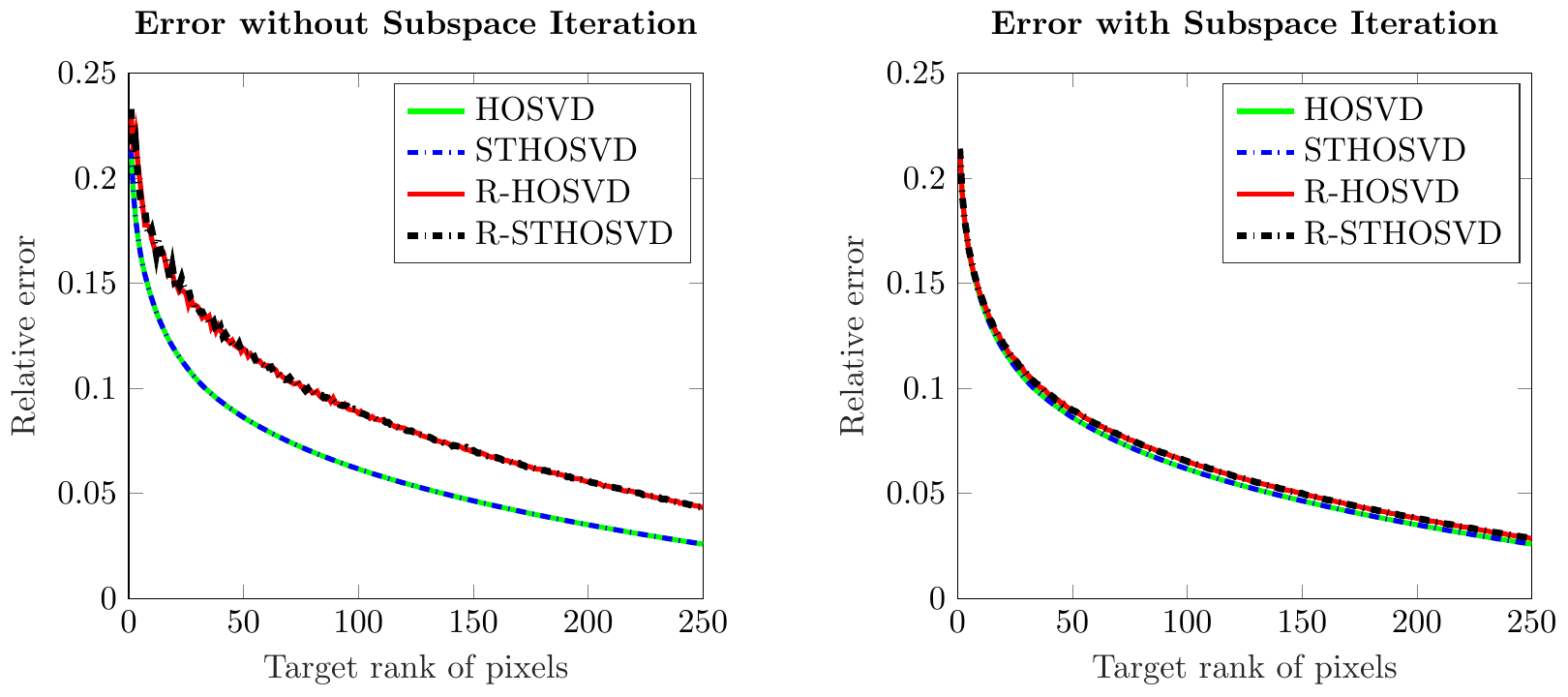}
\caption{Relative error for $\ten{X}$ with and without subspace iteration, compressing just the pixels (mode 2), as the target rank increases. The oversampling parameter was $p = 5$. The right hand plot was run with one step of subspace iteration.} 
\label{fig:faces_relerr}
\end{figure}

Now we consider the relative error if compressing two modes of the Olivetti dataset, namely the people and pixels (modes 1 and 2). We just compare the sequential algorithms here, as they are faster to run. The heat plots in \cref{fig:faces_heatplots} show the results.  We also compare this rank and error pair to that from the adaptive randomized STHOSVD algorithm (\cref{alg:adapt_sthosvd}). For all algorithms, the oversampling parameter was $p = 5$, and the processing order was $\rho = [2,1,*]$. The third mode is not compressed, indicated here by the asterisk.  We will elaborate on the adaptive algorithms in the next section.
\begin{figure}[!ht] \centering
    	\includegraphics[width=\textwidth]{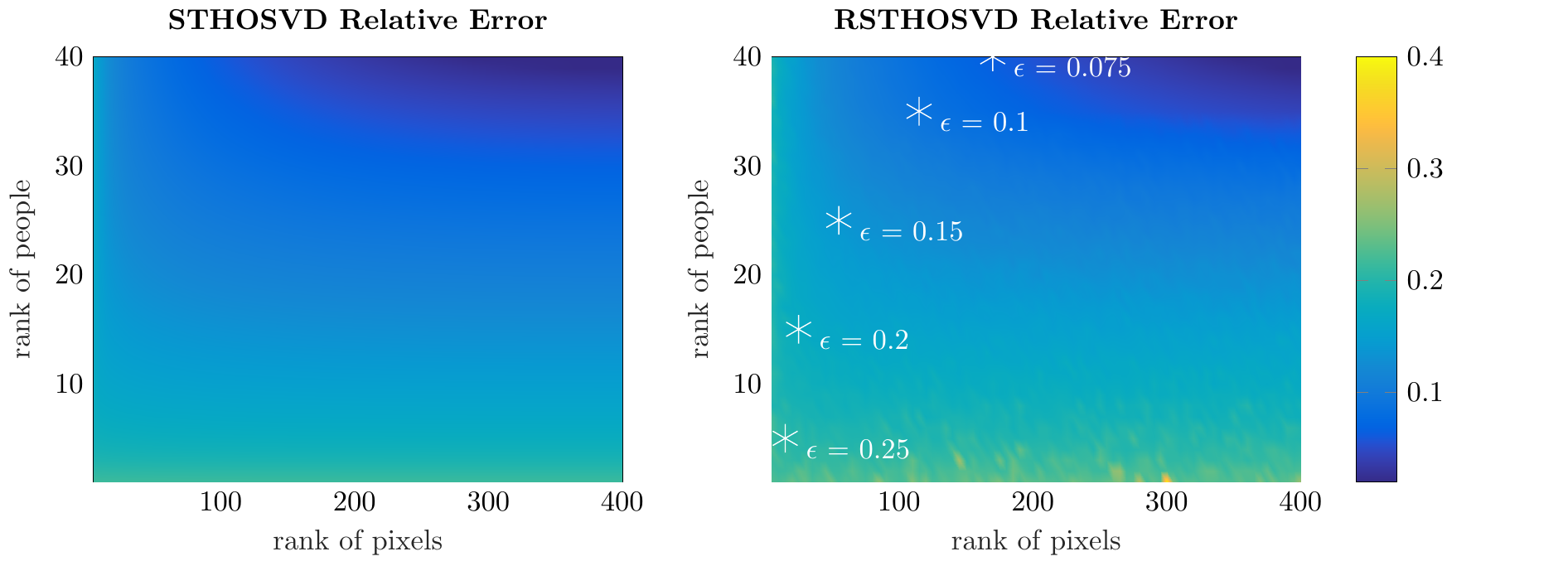}
\caption{Relative error for $\ten{X}$ as the target rank increases using the STHOSVD, compressing the pixels and the people (modes 2 and 1), plotted with the rank given by the Adaptive R-STHOSVD (\cref{alg:adapt_sthosvd}) with the desired relative error tolerance $\epsilon$. The processing order was $\rho = [2,1,*]$, and the oversampling parameter was $p=5$} 
\label{fig:faces_heatplots}
\end{figure}

 \subsubsection{Adaptive algorithms} We now evaluate our adaptive \cref{alg:adapt_rhosvd,alg:adapt_sthosvd} to compare the ranks given by inputing different relative error tolerances as the size of each dimension increases.  Taking the Hilbert tensor $\ten{X}$ defined in \cref{eqn:hilbert} with $d=3$, we give as input relative error tolerance $\epsilon$, processing order $\rho = [1,2,3]$, and blocking integer $b = 1$.  The size of the core tensor obtained from the adaptive STHOSVD algorithm is shown in \cref{tab:hilbert_adapt}.  In the $I_j = 25, \, j=1,2,3$ case, we can see that the error and corresponding rank are close to those shown in the left-hand plot of \cref{fig:hilbert_relerr}.
\begin{table}[!ht] \centering
\begin{tabular}{c  c | c | c | c | c | c }
\multicolumn{1}{c}{}  &   \multicolumn{6}{c}{tolerance $\epsilon$}\\
\multicolumn{1}{c}{}  &   & $10^{-3}$ & $10^{-4}$ & $10^{-5}$ & $10^{-6}$ & $10^{-7}$ \\ \cline{2-7}
& 25 & $(4,4,4)$ & $(5,5,5)$ & $(6,6,6)$ & $(7,7,7)$ & $(8,8,8)$  \\ 
$I$ & 50 & $(5,5,5)$ & $(6,6,6)$ & $(7,7,7)$ & $(8,8,8)$ & $(9,9,9)$ \\ 
$1\leq j \leq d$& 100 & $(5,5,5)$ & $(6,6,6)$ & $(8,8,8)$ & $(9,9,9)$ & $(10,10,10)$
\end{tabular}
\caption{Rank (size of the core tensor) obtained by the adaptive STHOSVD algorithm (\cref{alg:adapt_sthosvd}) with different relative error tolerances $\epsilon$ as the size of each dimension of $\ten{X}$ increase. The inputs are $\epsilon$, block size $b = 1$, and processing order $\rho = [1,2,3]$.}
\label{tab:hilbert_adapt}
\end{table}

Next, we consider the Olivetti dataset.  To compare our randomized algorithms to the standard algorithms, we give a desired relative error tolerance $\epsilon$ to the adaptive R-STHOSVD and find the size of the core tensor.  We then compare the actual error from that rank to the error from STHOSVD run with the same rank, and display the results in \cref{tab:faces_adapt}. The theory implies that we should see a smaller error from the STHOSVD results, and we observe that the error is only slightly smaller. Both algorithms were run with processing order $\rho = [2,1,3]$.

\begin{table}[!ht] \centering
\begin{tabular}{c|c|c|c}
Error tolerance $\epsilon$ & Corresponding rank $\mat{r}$ & Actual error & Rank-$\mat{r}$ STHOSVD error* \\
\hline
$0.25$ & $(3,10,1)$ & $0.1995$ & $0.1995$ \\   
$0.2$ & $(10,23,1)$ & $0.1799$ & $0.1796$ \\		
$0.15$ & $(22,51,5)$ & $0.1421$ & $0.1403$ \\		
$0.1$ & $(32,114,8)$ & $0.0965$ & $0.0946$ \\		
$0.05$ & $(38,237,10)$ & $0.0400$ & $0.0381$ \\		
$0.01$ & $(40,381,10)$ & $0.0057$ & $0.0055$
\end{tabular}
\caption{A comparison of the adaptive R-STHOSVD algorithm \cref{alg:adapt_sthosvd} to the STHOSVD. We first obtained the rank of the core tensor with the requested relative error tolerance from the adaptive algorithm. Then we compared the actual error of the approximation from the adaptive R-STHOSVD to that of an STHOSVD with the same rank.  The processing order for all runs was $\rho = [2,1,3]$.  *Each STHOSVD was computed with the corresponding rank found in the second column.}
\label{tab:faces_adapt}
\end{table}

\subsubsection{Algorithms for Sparse Tensors}
We now test our algorithms on sparse tensors.  First consider the synthetic sparse tensor $\ten{X}$ defined in \cref{eqn:sparse}.  For three different $\gamma$ values $\gamma = 2,10,200$, we compare the SP-STHOSVD algorithm to the STHOSVD and R-STHOSVD algorithms by plotting the relative error as the target rank $(r,r,r)$ increases.  Note that we are only comparing to the sequential algorithms in \cref{fig:synsparse_err}. This is because we have already shown the HOSVD and R-HOSVD algorithms to have a similar error with a higher cost in a problem of this size.  As inputs to our test algorithms, we used oversampling parameter $p=5$ and processing order $\rho = [1,2,3]$. We can see that the error for the sparse algorithm is slightly higher for the lower $r$ values, which is better than the expected result given the error bound in \cref{thm:pass_error}. 

\begin{figure}[!ht]
    \centering
    \includegraphics[width = \textwidth]{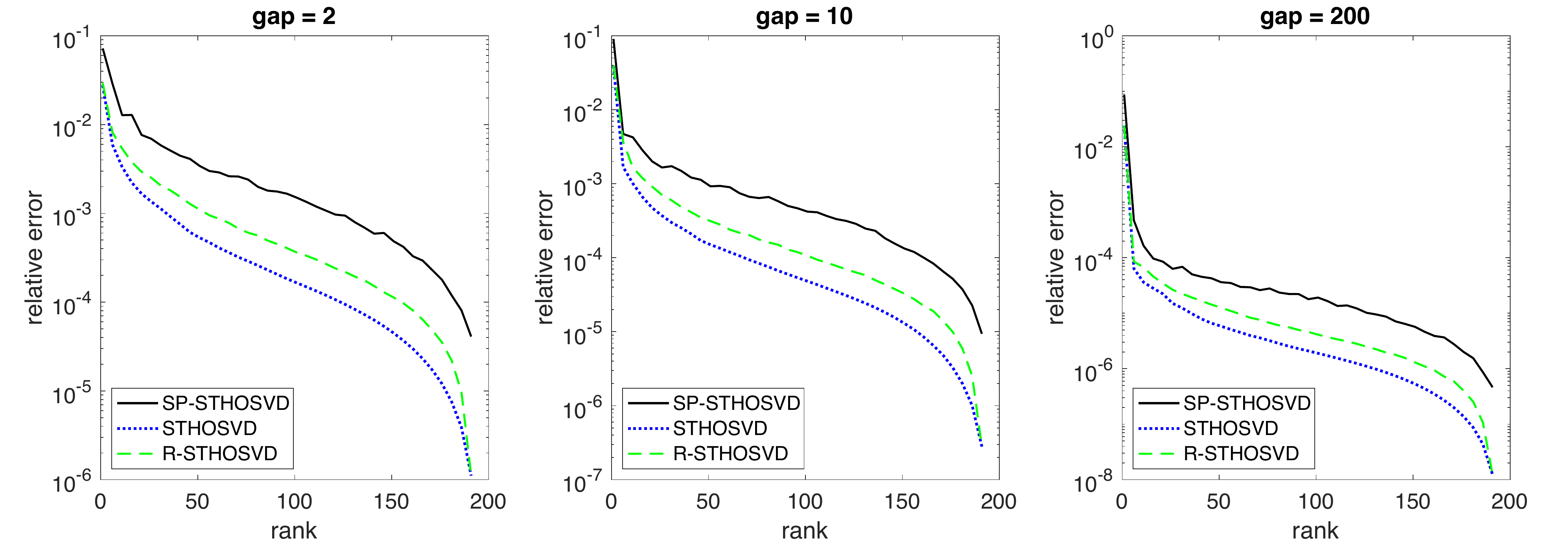}
    \caption{Relative error for synthetic sparse tensor $\ten{X}$ defined in \cref{eqn:sparse} with $\gamma = 2,10,200$ as the target rank $(r,r,r)$ increases. We compare the SP-STHOSVD algorithm (\cref{alg:sparse}) to the STHOSVD and R-STHOSVD algorithms with inputs of oversampling parameter $p = 5$ and processing order $\rho = [1,2,3]$.}
    \label{fig:synsparse_err}
\end{figure}

We now test our SP-STHOSVD algorithm on the real-world sparse tensors. Note that this algorithm oversamples but does not truncate, meaning that the rank of the resulting approximation given target rank $(r,r,r)$ will be $(r+p,r+p,r+p)$.  Then to compare the approximations of SP-STHOSVD to any of the other fixed rank algorithms presented thus far, we must use $(r+p,r+p,r+p)$ as the target rank for the fixed rank algorithms. 

To compare our algorithms, we ran the SP-STHOSVD and the R-STHOSVD algorithms on the condensed NELL-2 tensor (details in \cref{tab:frostt}) with inputs oversampling parameter $p = 5$, and processing order $\rho = [3,1,2]$ as the target rank $(r,r,r)$ increased.  The relative errors obtained are shown in \cref{tab:frostt_err}.  We can see that the error for SP-STHOSVD is higher than that of the R-STHOSVD, but this was anticipated from the theory.  We also compare the runtime of these two algorithms averaged over three runs to see their respective costs, which are also shown in \cref{tab:frostt_err}. 

\begin{table}[!ht]
    \centering
    \begin{tabular}{c|c|c|c|c}
        \hline
        \multicolumn{1}{c}{} & \multicolumn{3}{c}{NELL-2} & \multicolumn{1}{c}{} \\
        \hline
        \multicolumn{1}{c}{}  &   \multicolumn{2}{c}{Relative Error} & \multicolumn{2}{c}{Runtime in seconds}\\
        Target Rank  & SP-STHOSVD  & R-STHOSVD & SP-STHOSVD  & R-STHOSVD\\
        \hline
        $30$ & $0.2968$  & $0.1319$ & $0.5690$ & $17.0642$ \\
        $60$ & $0.2282$  & $0.0914$ &$0.9606$ & $18.9203$\\
        $90$ &  $0.1950$ & $0.0699$ & $1.5889$ & $23.3303$\\
        $120$ & $0.1666$ & $0.0573$ & $2.0706$ & $28.9399$\\
        $150$ & $0.1431$ & $0.0478$ & $2.1310$ & $ 33.6867$\\
        $180$ & $0.1201$ & $0.0417$ &  $2.3678$ & $39.0644$\\
        $210$  & $0.1181$ & $0.0367$ & $3.0832$ & $45.5227$\\
        $240$ & $0.1095$ &$0.0326$  & $3.7282$ & $52.4856$
    \end{tabular}
    \vspace{.4cm}
    
    \begin{tabular}{c|c|c|c|c}
        \hline
        \multicolumn{1}{c}{} & \multicolumn{3}{c}{Enron} & \multicolumn{1}{c}{} \\
        \hline
        \multicolumn{1}{c}{}  &   \multicolumn{2}{c}{Relative Error} & \multicolumn{2}{c}{Runtime in seconds}\\
        Target Rank  & SP-STHOSVD  & R-STHOSVD & SP-STHOSVD  & R-STHOSVD\\
        \hline
        $20$ & $0.6015$  & $0.2081$ & $0.4086$ & $31.5615$ \\
        $45$ & $0.3854$  & $0.1259$ &$0.7965$ & $34.5802$\\
        $70$ &  $0.3548$ & $0.0870$ & $1.3276$ & $36.6431$\\
        $95$ & $0.2038$ & $0.0632$ & $2.3465$ & $39.3095$\\
        $120$ & $0.1503$ & $0.0458$ & $2.8175$ & $ 39.7169$\\
        $145$ & $0.0976$ & $0.0332$ &  $3.5659$ & $42.0969$\\
        $170$  & $0.0756$ & $0.0239$ & $6.2158$ & $45.8429$\\
        $195$ & $0.0578$ &$0.0180$  & $6.8285$ & $50.2907$
    \end{tabular}
    \caption{The relative error and runtime of both SP-STHOSVD and R-STHOSVD on both the condensed and subsampled Enron dataset and the condensed NELL-2 dataset as the target rank $(r,r,r)$ increases.  The processing order was $\rho = [3,1,2]$, and the oversampling parameter was $p=5$. Note that the rank is the same for each mode for simplicity, and that the input rank for the R-STHOSVD was $(r+p,r+p,r+p)$ so the approximations have the same size.}
    \label{tab:frostt_err}
\end{table}

For the Enron dataset \cite{shetty2004enron}, we repeat the previous experiment, comparing the relative error and runtime of the SR-STHOSVD and R-STHOSVD as the target rank increases. The results are in \cref{tab:frostt_err}.  We see similar results as before, in that the error for SP-STHOSVD is higher than that of the R-STHOSVD, but the runtime is far less. 

\section{Conclusion}
In this paper, we makes several contributions in terms of new randomized algorithms and analysis for low-rank tensor decomposition in the Tucker format. Specifically, we proposed adaptive algorithms for problems where the target rank is not known beforehand and an algorithm that preserves the structure of the original tensor. We also provided probabilistic analysis of randomized compression algorithms, R-HOSVD and R-STHOSVD, as well as analysis for the newly proposed algorithms. We showed, through the analysis and numerical examples that using randomized techniques still allows for accurate approximations to tensors, and that the approximation error is comparable to deterministic algorithms, with much lower computational costs. 

\section{Acknowledgements} The authors would like to thank Ilse Ipsen for reading through the paper and giving useful feedback.
\bibliographystyle{abbrv}
\bibliography{ref}

\begin{thebibliography}{10}

\bibitem{olivettidataset}
{AT\&T Laboratories at Cambridge}.
\newblock Olivetti database of faces, 2002.

\bibitem{bader2007efficient}
B.~W. Bader and T.~G. Kolda.
\newblock Efficient {MATLAB} computations with sparse and factored tensors.
\newblock {\em SIAM Journal on Scientific Computing}, 30(1):205--231, December
  2007.

\bibitem{TTB_Sparse}
B.~W. Bader and T.~G. Kolda.
\newblock Efficient {MATLAB} computations with sparse and factored tensors.
\newblock {\em SIAM Journal on Scientific Computing}, 30(1):205--231, Dec.
  2007.

\bibitem{batselier2018computing}
K.~Batselier, W.~Yu, L.~Daniel, and N.~Wong.
\newblock Computing low-rank approximations of large-scale matrices with the
  tensor network randomized {SVD}.
\newblock {\em SIAM Journal on Matrix Analysis and Applications},
  39(3):1221--1244, 2018.

\bibitem{battaglino2018practical}
C.~Battaglino, G.~Ballard, and T.~G. Kolda.
\newblock A practical randomized {CP} tensor decomposition.
\newblock {\em SIAM Journal on Matrix Analysis and Applications},
  39(2):876--901, 2018.

\bibitem{biagioni2015randomized}
D.~J. Biagioni, D.~Beylkin, and G.~Beylkin.
\newblock Randomized interpolative decomposition of separated representations.
\newblock {\em Journal of Computational Physics}, 281:116--134, 2015.

\bibitem{carlson2010toward}
A.~Carlson, J.~Betteridge, B.~Kisiel, B.~Settles, E.~R. Hruschka~Jr., and T.~M.
  Mitchell.
\newblock Toward an architecture for never-ending language learning.
\newblock In {\em AAAI}, volume~5, page~3, 2010.

\bibitem{che2018randomized}
M.~Che and Y.~Wei.
\newblock Randomized algorithms for the approximations of {Tucker} and the
  {Tensor Train} decompositions.
\newblock {\em Advances in Computational Mathematics}, pages 1--34, 2018.

\bibitem{cheng2005compression}
H.~Cheng, Z.~Gimbutas, P.-G. Martinsson, and V.~Rokhlin.
\newblock On the compression of low rank matrices.
\newblock {\em SIAM Journal on Scientific Computing}, 26(4):1389--1404, 2005.

\bibitem{cichocki2016tensor}
A.~Cichocki, N.~Lee, I.~Oseledets, A.-H. Phan, Q.~Zhao, D.~P. Mandic, et~al.
\newblock Tensor networks for dimensionality reduction and large-scale
  optimization: Part 1 low-rank tensor decompositions.
\newblock {\em Foundations and Trends{\textregistered} in Machine Learning},
  9(4-5):249--429, 2016.

\bibitem{cichocki2017tensor}
A.~Cichocki, A.-H. Phan, Q.~Zhao, N.~Lee, I.~Oseledets, M.~Sugiyama, D.~P.
  Mandic, et~al.
\newblock Tensor networks for dimensionality reduction and large-scale
  optimization: Part 2 applications and future perspectives.
\newblock {\em Foundations and Trends{\textregistered} in Machine Learning},
  9(6):431--673, 2017.

\bibitem{de2000multilinear}
L.~De~Lathauwer, B.~De~Moor, and J.~Vandewalle.
\newblock A multilinear singular value decomposition.
\newblock {\em SIAM journal on Matrix Analysis and Applications},
  21(4):1253--1278, 2000.

\bibitem{de2000best}
L.~De~Lathauwer, B.~De~Moor, and J.~Vandewalle.
\newblock On the best rank-$1$ and rank-$(r_1, r_2,..., r_n)$ approximation of
  higher-order tensors.
\newblock {\em SIAM journal on Matrix Analysis and Applications},
  21(4):1324--1342, 2000.

\bibitem{drineas2007randomized}
P.~Drineas and M.~W. Mahoney.
\newblock A randomized algorithm for a tensor-based generalization of the
  singular value decomposition.
\newblock {\em Linear algebra and its applications}, 420(2-3):553--571, 2007.

\bibitem{drineas2016randnla}
P.~Drineas and M.~W. Mahoney.
\newblock {RandNLA}: randomized numerical linear algebra.
\newblock {\em Communications of the ACM}, 59(6):80--90, 2016.

\bibitem{drmac2016new}
Z.~Drma\v{c} and S.~Gugercin.
\newblock A new selection operator for the discrete empirical interpolation
  method---improved a priori error bound and extensions.
\newblock {\em SIAM Journal on Scientific Computing}, 38(2):A631--A648, 2016.

\bibitem{drmac2018discrete}
Z.~Drma\v{c} and A.~K. Saibaba.
\newblock The discrete empirical interpolation method: Canonical structure and
  formulation in weighted inner product spaces.
\newblock {\em SIAM Journal on Matrix Analysis and Applications},
  39(3):1152--1180, 2018.

\bibitem{eckart1936approximation}
C.~Eckart and G.~Young.
\newblock The approximation of one matrix by another of lower rank.
\newblock {\em Psychometrika}, 1(3):211--218, 1936.

\bibitem{erichson2017randomized}
N.~B. Erichson, K.~Manohar, S.~L. Brunton, and J.~N. Kutz.
\newblock Randomized {CP} tensor decomposition.
\newblock {\em arXiv preprint arXiv:1703.09074}, 2017.

\bibitem{grasedyck2013literature}
L.~Grasedyck, D.~Kressner, and C.~Tobler.
\newblock A literature survey of low-rank tensor approximation techniques.
\newblock {\em GAMM-Mitteilungen}, 36(1):53--78, 2013.

\bibitem{gu2015subspace}
M.~Gu.
\newblock Subspace iteration randomization and singular value problems.
\newblock {\em SIAM Journal on Scientific Computing}, 37(3):A1139--A1173, 2015.

\bibitem{gu1996efficient}
M.~Gu and S.~C. Eisenstat.
\newblock Efficient algorithms for computing a strong rank-revealing {QR}
  factorization.
\newblock {\em SIAM Journal on Scientific Computing}, 17(4):848--869, 1996.

\bibitem{hackbusch2012tensor}
W.~Hackbusch.
\newblock {\em Tensor spaces and numerical tensor calculus}, volume~42.
\newblock Springer Science \& Business Media, 2012.

\bibitem{halko2011finding}
N.~Halko, P.-G. Martinsson, and J.~A. Tropp.
\newblock Finding structure with randomness: Probabilistic algorithms for
  constructing approximate matrix decompositions.
\newblock {\em SIAM review}, 53(2):217--288, 2011.

\bibitem{horn1990matrix}
R.~A. Horn and C.~R. Johnson.
\newblock {\em Matrix analysis}.
\newblock Cambridge university press, 1990.

\bibitem{huber2017randomized}
B.~Huber, R.~Schneider, and S.~Wolf.
\newblock A randomized {Tensor Train} singular value decomposition.
\newblock In {\em Compressed Sensing and its Applications}, pages 261--290.
  Springer, 2017.

\bibitem{jacod2012probability}
J.~Jacod and P.~Protter.
\newblock {\em Probability essentials}.
\newblock Springer Science \& Business Media, 2012.

\bibitem{kolda2009tensor}
T.~G. Kolda and B.~W. Bader.
\newblock Tensor decompositions and applications.
\newblock {\em SIAM review}, 51(3):455--500, 2009.

\bibitem{kressner2017recompression}
D.~Kressner and L.~Perisa.
\newblock Recompression of {H}adamard products of tensors in {T}ucker format.
\newblock {\em SIAM Journal on Scientific Computing}, 39(5):A1879--A1902, 2017.

\bibitem{mahoney2011randomized}
M.~W. Mahoney et~al.
\newblock Randomized algorithms for matrices and data.
\newblock {\em Foundations and Trends{\textregistered} in Machine Learning},
  3(2):123--224, 2011.

\bibitem{mahoney2008tensor}
M.~W. Mahoney, M.~Maggioni, and P.~Drineas.
\newblock {Tensor-CUR} decompositions for tensor-based data.
\newblock {\em SIAM Journal on Matrix Analysis and Applications},
  30(3):957--987, 2008.

\bibitem{malik2019fast}
O.~A. Malik and S.~Becker.
\newblock Fast randomized matrix and tensor interpolative decomposition using
  {CountSketch}.
\newblock {\em arXiv preprint arXiv:1901.10559}, 2019.

\bibitem{martinsson2016randomized}
P.~G. Martinsson and S.~Voronin.
\newblock A randomized blocked algorithm for efficiently computing
  rank-revealing factorizations of matrices.
\newblock {\em SIAM Journal on Scientific Computing}, 38(5):S485--S507, 2016.

\bibitem{saibaba2017hoid}
A.~K. Saibaba.
\newblock {HOID}: higher order interpolatory decomposition for tensors based on
  {T}ucker representation.
\newblock {\em SIAM Journal on Matrix Analysis and Applications},
  37(3):1223--1249, 2016.

\bibitem{saibaba2019randomized}
A.~K. Saibaba.
\newblock Randomized subspace iteration: Analysis of canonical angles and
  unitarily invariant norms.
\newblock {\em SIAM Journal on Matrix Analysis and Applications}, 40(1):23--48,
  2019.

\bibitem{shetty2004enron}
J.~Shetty and J.~Adibi.
\newblock The {E}nron email dataset database schema and brief statistical
  report.
\newblock {\em Information sciences institute technical report, University of
  Southern California}, 4, 2004.

\bibitem{frosttdataset}
S.~Smith, J.~W. Choi, J.~Li, R.~Vuduc, J.~Park, X.~Liu, and G.~Karypis.
\newblock {FROSTT}: The formidable repository of open sparse tensors and tools,
  2017.

\bibitem{sorensen2016deim}
D.~C. Sorensen and M.~Embree.
\newblock A {DEIM} induced {CUR} factorization.
\newblock {\em SIAM Journal on Scientific Computing}, 38(3):A1454--A1482, 2016.

\bibitem{szyld2006many}
D.~B. Szyld.
\newblock The many proofs of an identity on the norm of oblique projections.
\newblock {\em Numerical Algorithms}, 42(3-4):309--323, 2006.

\bibitem{tsourakakis2010mach}
C.~E. Tsourakakis.
\newblock Mach: Fast randomized tensor decompositions.
\newblock In {\em Proceedings of the 2010 SIAM International Conference on Data
  Mining}, pages 689--700. SIAM, 2010.

\bibitem{vannieuwenhoven2012new}
N.~Vannieuwenhoven, R.~Vandebril, and K.~Meerbergen.
\newblock A new truncation strategy for the {higher}-order singular value
  decomposition.
\newblock {\em SIAM Journal on Scientific Computing}, 34(2):A1027--A1052, 2012.

\bibitem{vasilescu2002multilinear}
M.~A.~O. Vasilescu and D.~Terzopoulos.
\newblock Multilinear analysis of image ensembles: Tensorfaces.
\newblock In {\em European Conference on Computer Vision}, pages 447--460.
  Springer, 2002.

\bibitem{vervliet2016randomized}
N.~Vervliet and L.~De~Lathauwer.
\newblock A randomized block sampling approach to canonical polyadic
  decomposition of large-scale tensors.
\newblock {\em IEEE Journal of Selected Topics in Signal Processing},
  10(2):284--295, 2016.

\bibitem{vervliet2016tensorlab}
N.~Vervliet, O.~Debals, L.~Sorber, M.~Van~Barel, and L.~De~Lathauwer.
\newblock {T}ensorlab 3.0. available online.
\newblock {\em URL: http://www. tensorlab. net}, 2016.

\bibitem{gu2016efficient}
W.~Yu, Y.~Gu, and Y.~Li.
\newblock Efficient randomized algorithms for the fixed-precision low-rank
  matrix approximation.
\newblock {\em SIAM Journal on Matrix Analysis and Applications},
  39(3):1339--1359, 2018.

\bibitem{zhang2016randomized}
J.~Zhang, A.~K. Saibaba, M.~E. Kilmer, and S.~Aeron.
\newblock A randomized tensor singular value decomposition based on the
  t-product.
\newblock {\em Numerical Linear Algebra with Applications}, 25(5):e2179, 2018.

\bibitem{zhou2014decomposition}
G.~Zhou, A.~Cichocki, and S.~Xie.
\newblock Decomposition of big tensors with low multilinear rank.
\newblock {\em arXiv preprint arXiv:1412.1885}, 2014.

\end{thebibliography}
\end{document}